\documentclass[onefignum,onetabnum]{siamonline220329}

\usepackage{lipsum}
\usepackage{amsfonts}
\usepackage{graphicx}
\usepackage{epstopdf}
\usepackage{algorithmic}
\ifpdf
  \DeclareGraphicsExtensions{.eps,.pdf,.png,.jpg}
\else
  \DeclareGraphicsExtensions{.eps}
\fi

\usepackage{enumitem}
\setlist[enumerate]{leftmargin=.5in,label={\textup{(\roman*)}}}
\setlist[itemize]{leftmargin=.5in}

\usepackage{faktor}

\newsiamremark{remark}{Remark}
\newsiamremark{hypothesis}{Hypothesis}
\crefname{hypothesis}{Hypothesis}{Hypotheses}
\newsiamthm{claim}{Claim}

\usepackage{amsmath}
\usepackage{amsfonts}
\usepackage{mathtools}
\usepackage{yhmath}

\usepackage{picinpar,moresize,xfrac,graphpap,dcolumn,wrapfig,graphicx}

\usepackage{subcaption}
\usepackage{stackengine}

\usepackage{tikz,pgfplots}
\usepackage{tikz-cd}
\usepackage{pgf}
\usepgfplotslibrary{colormaps}
\usetikzlibrary{pgfplots.colormaps}
\usetikzlibrary{arrows,automata,positioning,backgrounds}
\usepackage{rotating}
\pgfplotsset{compat=newest}
\pgfplotsset{plot coordinates/math parser=false}
\newlength\figureheight
\newlength\figurewidth
\usepackage{soul,array,calc,url,ragged2e,graphpap}
\usepackage{booktabs} % better tabular
\usepackage{tabularx}
\usepackage{colortbl}
\usepackage{multirow}
\usepackage{hyphenat}
\usepackage{transparent}
\usepackage{mathrsfs}
\usepackage{mathtools}
\mathtoolsset{centercolon}
\usepackage{nicefrac}
\usepackage{units}
\usepackage[normalem]{ulem}
\usepackage{cancel}
\usepackage{pbox}
\usepackage{multicol}
\usepackage{cases}
\usepackage[all]{xy}
\usepackage{nicematrix}

\usepackage{cleveref}
\crefmultiformat{subequation}%
{\edef\crefstripprefixinfo{#1}(#2#1#3}%
{,#2\crefstripprefix{\crefstripprefixinfo}{#1}#3)}%
{,#2\crefstripprefix{\crefstripprefixinfo}{#1}#3}%
{,#2\crefstripprefix{\crefstripprefixinfo}{#1}#3)}

\makeatletter
\DeclareFontFamily{OMX}{MnSymbolE}{}
\DeclareSymbolFont{MnLargeSymbols}{OMX}{MnSymbolE}{m}{n}
\SetSymbolFont{MnLargeSymbols}{bold}{OMX}{MnSymbolE}{b}{n}
\DeclareFontShape{OMX}{MnSymbolE}{m}{n}{
    <-6>  MnSymbolE5
   <6-7>  MnSymbolE6
   <7-8>  MnSymbolE7
   <8-9>  MnSymbolE8
   <9-10> MnSymbolE9
  <10-12> MnSymbolE10
  <12->   MnSymbolE12
}{}
\DeclareFontShape{OMX}{MnSymbolE}{b}{n}{
    <-6>  MnSymbolE-Bold5
   <6-7>  MnSymbolE-Bold6
   <7-8>  MnSymbolE-Bold7
   <8-9>  MnSymbolE-Bold8
   <9-10> MnSymbolE-Bold9
  <10-12> MnSymbolE-Bold10
  <12->   MnSymbolE-Bold12
}{}
\let\llangle\@undefined
\let\rrangle\@undefined
\DeclareMathDelimiter{\llangle}{\mathopen}%
                     {MnLargeSymbols}{'164}{MnLargeSymbols}{'164}
\DeclareMathDelimiter{\rrangle}{\mathclose}%
                     {MnLargeSymbols}{'171}{MnLargeSymbols}{'171}
\makeatother

\newtheorem{example}{Example}

\newcommand{\figref}[1]{\textup{Fig.~\ref{#1}}}

\newcommand{\teqref}[1]{\textup{Eq.~(\ref{#1})}}

\newcommand{\secref}[1]{\textup{Section~\ref{#1}}}

\newcommand{\thmref}[1]{\textup{Theorem~\ref{#1}}}

\newcommand{\corref}[1]{\textup{Corollary~\ref{#1}}}

\def\ie{\emph{i.e.}}
\def\eg{\emph{e.g.}}

\def\resp{resp.}
\def\RR{\mathbb{R}}

\def\ZZ{\mathbb{Z}}

\usepackage{bbm}
\DeclareSymbolFont{bbold}{U}{bbold}{m}{n}
\DeclareSymbolFontAlphabet{\mathbbold}{bbold}
\def\im {\operatorname{im}}
\def\ker{\operatorname{ker}}

\newcommand{\grad}{\mathop{\mathrm{grad}}\nolimits}

\newcommand{\curl}{\mathop{\mathrm{curl}}\nolimits}

\renewcommand{\div}{\mathop{\mathrm{div}}\nolimits}
\newcommand{\LD}{\mathop{\mathscr{L}}\nolimits}

\def\tt{\mathbbold{t}}

\def\ip{\iota}

\def\bowB{\wideparen{B}}
\def\bowmu{\wideparen{\mu}}
\def\bowg{\wideparen{g}}
\def\bowstar{\mathbin{\wideparen{\star}}}

\def\barB{\overline{B}}
\def\barmu{\overline{\mu}}
\def\barg{\overline{g}}
\def\barstar{\mathbin{\overline{\star}}}

\def\bareta{\overline{\eta}}

\headers{Force-Free Fields are Conformally Geodesic}{A. Chern and O. Gross}%, 
\title{Force-Free Fields are Conformally Geodesic%\thanks{Submitted to the editors \today.}
}%}

\author{
  Albert Chern%
\thanks{
    University of California, San Diego, 9500 Gilman Dr, MC 0404 La Jolla, CA 92093-0404
  (\email{alchern@ucsd.edu}).}
\and
  Oliver Gross%
  \thanks{
  Technische Universit{\"a}t Berlin, %Institut für Mathematik, 
  Stra\ss e des 17. Juni 136, 10623, Berlin, Germany 
  (\email{ogross@math.tu-berlin.de}%, 
  ).}
  }

\ifpdf
\hypersetup{
  pdftitle={Force-Free Fields are Conformally Geodesic},
  pdfauthor={A. Chern and O. Gross}
}
\fi

\begin{document}

\maketitle

% REQUIRED
% In this paper, we establish an equivalence between force-free fields and conformally geodesic fields, and between harmonic fields and conformally eikonal fields in the context of conformal geometry. Specifically, we relate stationary points of hierarchies of \(L^2\) \resp\@ \(L^1\)-optimization problems---distinguished by the topological constraints they impose---by a conformal change of metric. We provide an explicit construction of the conformal factors relating the relevant metrics and identify the field lines of physical vector fields fields as conformal geodesics. Despite the allowed topological complexity of the fields under consideration, which may even be chaotic, these observations reveal geometric order in them which is obtained by merely rescaling the metric. 
\begin{abstract}
In this paper, we establish an equivalence between force-free fields and conformally geodesic fields, and between harmonic fields and conformally eikonal fields. In contrast to previous work, our approach and equivalence results generalize to arbitrary dimension \(n\geq 3\). In accordance with three-dimensional theory, our defining equations emerge as the Euler-Lagrange equations of hierarchies of variational principles---distinguished by the topological constraints they impose---and retain the known inclusions of the special cases from each other. Specifically, we relate stationary points of hierarchies of \(L^2\) \resp\@ \(L^1\)-optimization problems by a conformal change of metric, provide an explicit construction of the conformal factors relating the relevant metrics and identify the field lines of physical vector fields fields as conformal geodesics. %\OG{\sout{Despite the allowed topological complexity of the fields under consideration, these observations reveal geometric order which is obtained by merely pointwise rescaling of the metric.} R1 probably does not like this one}
%Moreover, we relate our results to Killing vector fields with constant length. 
\end{abstract}

% REQUIRED

% REQUIRED

% 30C70: Extremal problems for conformal and quasiconformal mappings, variational methods
% 30F45: Conformal metrics (hyperbolic, Poincaré, distance functions)
% 53A30: Conformal differential geometry
% 81T40: Two-dimensional field theories, conformal field theories, etc.
% 49K30: Optimal solutions belonging to restricted classes
% 53C38: Calibrations and calibrated geometries
% 53C22: Geodesics
% 53C80: Applications to physics
% 53C65: Integral geometry; differential forms, currents, etc.
% 53C12: Foliations (differential geometric aspects)
% 53A30: Conformal differential geometry
% 58E30: Variational principles

\section{Introduction}\label{sec:Introduction}
Geometric structures are key structural motifs in a multitude of natural systems ranging from molecules and polymers to the field lines of fluid flows or electromagnetic fields in plasma. Therefore, understanding these structures is an important task in both mathematics and the natural sciences~\cite{cantarella1999topological, GM1999:GCT,Yeates_2014, Yeates:2018:GNP, moffatt1969degree, Moffatt2014PNAS}. In the study of variational problems for field lines that foliate a space, there are two thoroughly explored yet relatively disjoint pillars of focuses. The first, arising naturally in plasma physics and hydrodynamics, concerns force-free fields, while the second, particularly relevant in Riemannian geometry and optical physics, involves field lines as geodesics. Remarkably, we demonstrate that these seemingly distinct classes of flux fields share a direct relationship within the framework of conformal geometry.

In three dimensions, force-free fields, equivalently referred to as Beltrami fields, are vector fields \(B\) satisfying \((\curl B)\times B = 0\) and \(\div B = 0\).  In plasma physics, these force-free fields correspond to magnetic fields that produce zero Lorentz force.  They are extensively investigated in solar physics and controlled fusion since they constitute static plasma states with negligible pressure~\cite{Priest:2014:MHD, bellan2018magnetic, Yeates:2018:GNP}. In the realm of fluid dynamics, force-free fields are known as Beltrami velocity fields and constitute special steady solutions to the incompressible Euler equations~\cite{arnold2008topological}.  Force-free fields also include harmonic fields (\(\curl B = 0\), \(\div B = 0\)) as a significant subclass that plays central roles in vacuum electromagnetism, hydrodynamics, and the general theory of vector fields.

On the other hand, geodesic foliations are characterized by vector fields whose integral curves form geodesics \cite{Deshmukh:2020:GVF}. These foliations describe optical paths according to Fermat's principle. A special subclass of geodesic vector fields consists of gradients of distance functions, termed \emph{eikonal fields}.  These fields correspond to solutions to Beckmann optimal transport problems \cite{Santambrogio:2015:OT,Brezis:2019:PP, Deshmukh:2019:GVF}, untwisted light fields with applications in caustic designs \cite{Schwartzburg:2014:HCC}, and calibrated forms in calibrated geometry \cite{Harvey:1982:CG}. 

A natural question to ask is: \emph{``given a vector field on a manifold, does there exists a Riemannian metric such that the field lines form a geodesic foliation?''} Necessary and sufficient conditions for an affirmative answer have been given in, \eg, \cite{Gluck:79:OL, gluck2006dynamical} or \cite{sullivan1978foliation}. So-called \emph{geodesible} vector fields have been studied in numerous contexts. For example, they are of interest in the context of adaptions of the \emph{Seifert conjecture} or \emph{Weinstein conjecture} and relate to, \eg, Reeb vector fields on contact manifolds, stable Hamiltonian structures or Beltrami fields~\cite{Etnyre:2000:CTH, Rechtman2009UAD, Cieliebak2015, cardona2021geometry}. 

A generalized concept of geodesic fields is the notion of \emph{conformally geodesic fields} \cite{Fialkow:1939:CG,Fialkow:1942:CTC,Dunajski:2021:VPC}, which are fields that become geodesic after some conformal change of metric. Conformal geodesic fields can depict optical paths in a medium with a non-uniform index of refraction~\cite{Schwartzburg:2014:HCC}. We refer to~\cite{Gross:2024:CGIMHD} for an in depth overview over how all of these fields are related.
\begin{remark}
     Some authors use the term \emph{conformal geodesic} for vector fields which satisfy \(\nabla_X X = f X\) for some scalar function \(f\), \ie, whose integral curves are geodesics up to reparametrization. However, we follow the notion of conformal geodesic coined by, \eg, Fialkow \cite{Fialkow:1939:CG} as fields whose integral curves are geodesic after a conformal change of ambient metric, rather than merely along the integral curve. 
\end{remark}

The main results of this paper are equivalence theorems between the two classes:
\begin{theorem}
    \label{thm:MainTheorem1}
    Force-free fields are conformally geodesic.
\end{theorem}
\begin{theorem}
    \label{thm:MainTheorem2}
    Harmonic fields are conformally eikonal.
\end{theorem}
These theorems can be expressed as statements about field lines on an \(n\)-dimensional conformal manifold. 
In the absence of a specific metric, field lines are merely represented by a closed \((n-1)\)-form \(\beta\) (equivalently, a 1-current), referred to as a flux form.%

 \begin{remark} We acknowledge that magnetic fields in dimensions other than 3 should remain as 2-forms instead of \((n-1)\)-forms, since they arise as the curvature of a connection in the context of \(U(1)\) gauge theory~\cite[Sec. 10.5.1]{nakahara2018geometry}.  However, this paper's primary focus is on flux forms that describe field lines foliating a space.
 \end{remark}
 Each metric within the conformal class enables a vector field representation of the flux form, as well as examinations of metric-dependent qualities such as being geodesic or being force-free. \cref{thm:MainTheorem1,thm:MainTheorem2} assert that \emph{a flux form admits a metric with respect to which it is force-free (\resp\@ harmonic) if and only if it admits a (possibly different, but conformally equivalent) metric with respect to which it is geodesic (\resp\@ eikonal).} Our result extends previous results (see, \eg, \cite{Etnyre:2000:CTH, rechtman_2010}) in the sense that we can establish an explicit relation between the \resp\@ relevant metrics and in contrast to previous approaches (see, \eg, \cite{cardona2021geometry}) our definitions and results generalize to higher dimensions.

There are several significant implications from the equivalence theorems.  

%\paragraph
\subsection{Structures in Steady Euler Flow}
Steady Euler flows are governed by \(\nabla_BB + \grad p = 0\), which can be rewritten as \((\curl B)\times B + \grad b = 0\), where \(B\) is the divergence-free velocity field, \(p\) is the fluid pressure, and \(b = p + \frac{1}{2}|B|^2\) is the Bernoulli pressure.
In 1965, V.I.~Arnold \cite{Arnold:2014:TES} provided a sequence of structural theorems that describe the increasing complexity in a steady Euler flow. When \(\grad b\neq 0\), the fluid domain is decomposed into finitely many cells fibered into invariant tori or annuli given by the level sets of the Bernoulli pressure \(b\).  The flow lines generated by \(B\) are either all closed or all dense on each invariant surface.  When \(\grad b = 0\), the Bernoulli level sets no longer exist, and we obtain a Beltrami field \((\curl B)\times B = 0\), implying \(\curl B = \lambda B\) for some scalar function \(\lambda\).  By taking the divergence, we get \((\grad \lambda)\mathrel{\bot} B\), which implies that the flow \(B\) can still admit invariant surfaces given by the level sets of \(\lambda\), provided that \(\grad \lambda\neq 0\). 

\begin{figure}[h!]
    \centering
    \includegraphics[width = \textwidth]{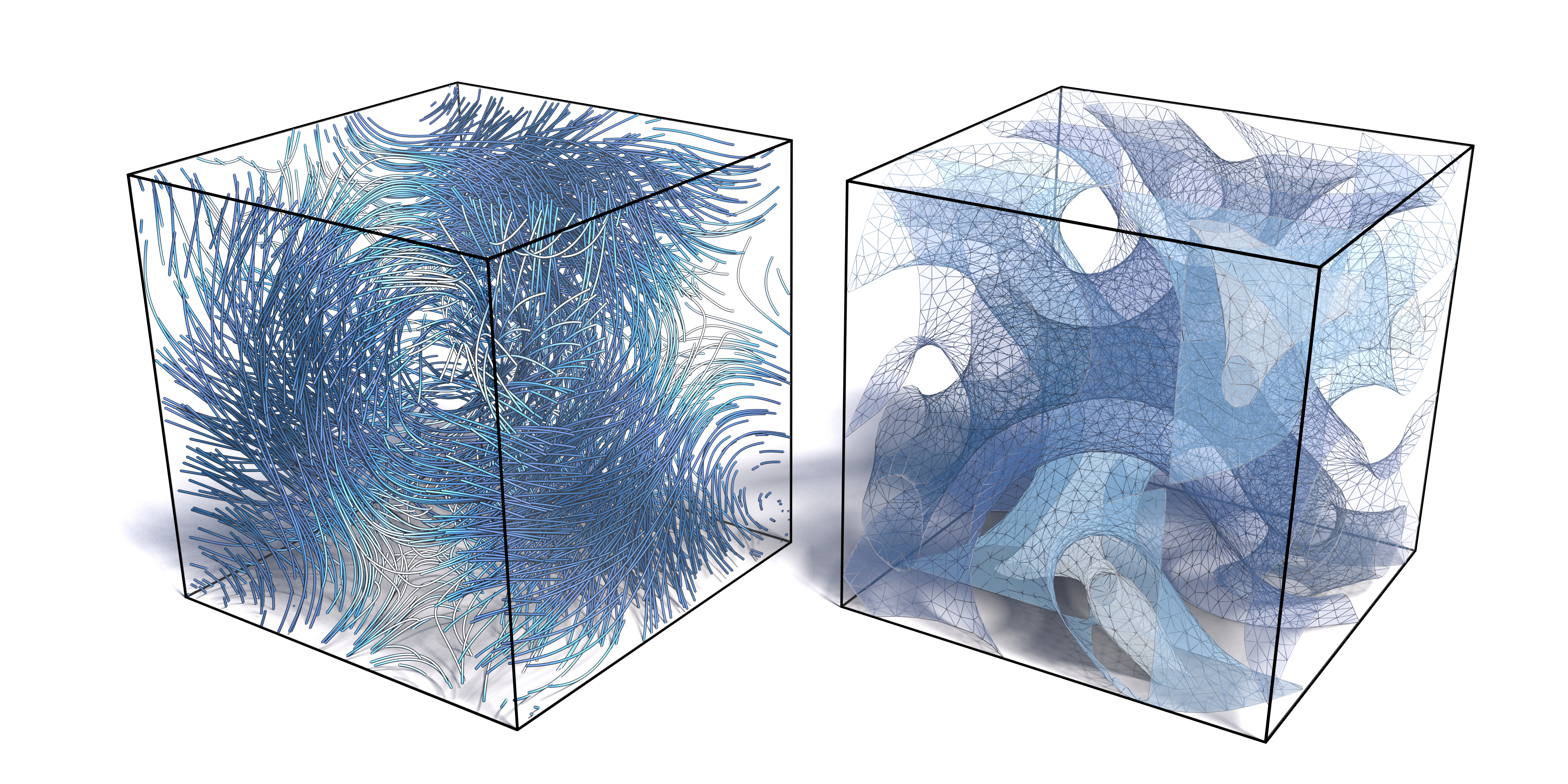}
    \caption{Left: Field lines of the ABC-flow \cref{eq:ABCFlow} for \(A=B=C=1\). Right: Levelsets of the squared magnitude of the flow field which, taken as a conformal factor, makes the field lines geodesic (\cref{thm:ABDCorollary}).}
    \label{fig:ABCFlow}
\end{figure}
If \(\lambda\) is a constant, then the flow lines for \(B\) become chaotic.  A popular example of such a Beltrami field with constant \(\lambda\) is the \emph{Arnold–Beltrami–Childress flow (ABC-flow, \cref{fig:ABCFlow})}, which  (on the three dimensional torus \((\nicefrac{\RR}{2\pi})^3\)) satisfies 
\begin{align}
    \label{eq:ABCFlow}
    \begin{bmatrix}
        \dot x \\ \dot y \\ \dot z
    \end{bmatrix}=\begin{bmatrix}
        A\sin(z) + C\cos(y) \\ B\sin(x) + A\cos(z) \\ C\sin(y) + B\cos(x)
    \end{bmatrix}
\end{align}
for parameters \(A,B,C\in\RR\) and is known to exhibit chaotic streamlines.  For a survey of this topic we refer the reader to \cite[Ch.~2.1]{arnold2008topological} and \cite{enciso2016beltrami, Berger:2023:SEF}.
\begin{corollary}\label{thm:ABDCorollary}
    ABC-flows are conformally geodesic.
\end{corollary}

%\paragraph
\subsection{Solar Coronal Loops}
The solar atmosphere is filled with magnetic fields that form arches connecting positive and negative surface magnetic fluxes.  In more active regions of sun's surface, the magnetic fields concentrate into strong and often twisted flux ropes connecting sunspots.  These flux ropes are generally modeled by force-free magnetic fields.  In quiet regions of the solar surface, the magnetic fields are relaxed to harmonic fields.  A popular model for a harmonic magnetic field in the solar atmosphere is known as the \emph{potential-field source surface} model.  In particular, one observes an absence of twisted magnetic fields in these quiet regions as the twists have been resolved through dissipative reconnection events over a longer period of relaxation time~\cite{Yeates2020}.
\begin{figure}[h]
    \centering
    \includegraphics[width = .85\textwidth]{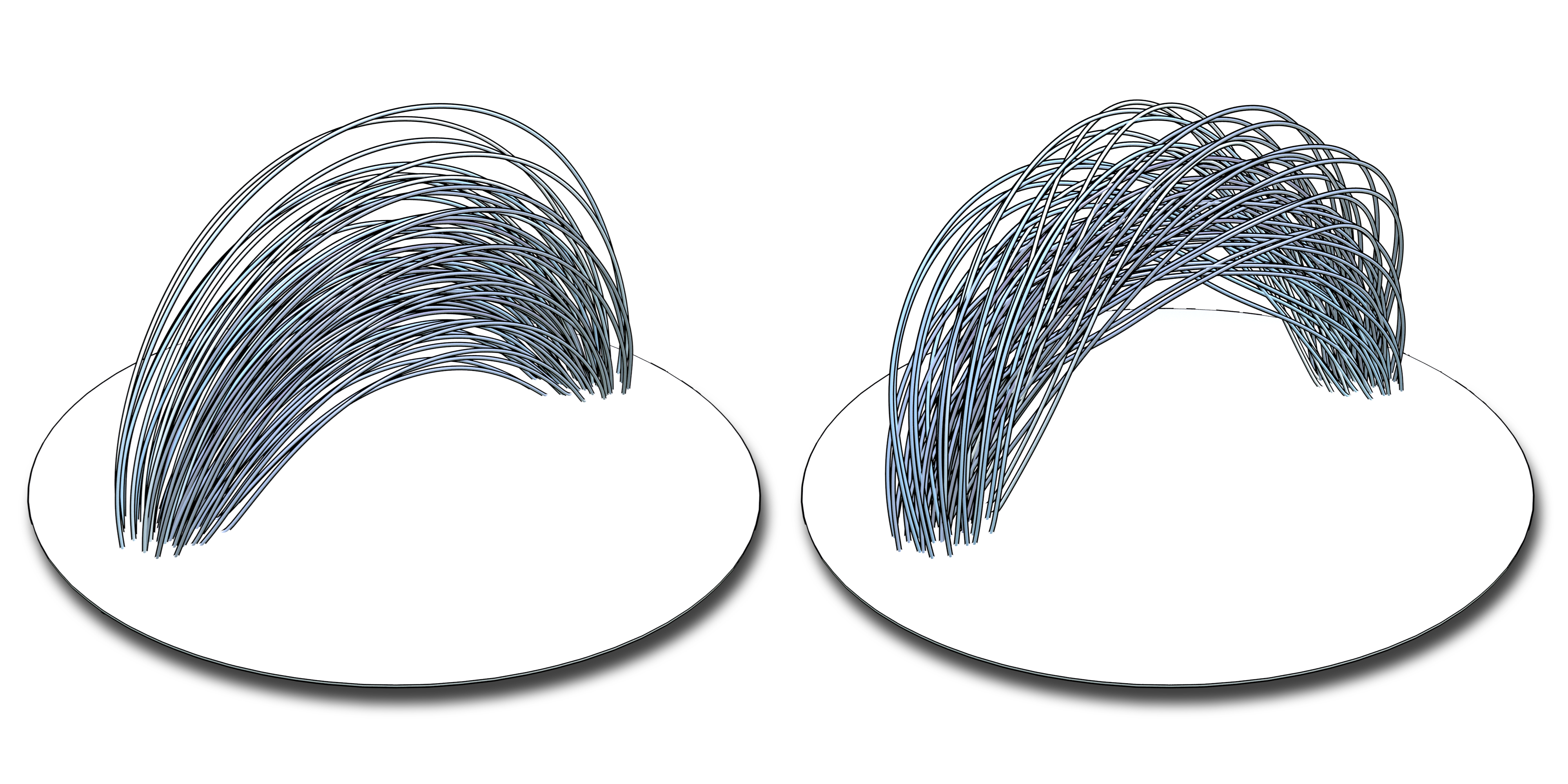}
    \caption{In a static equilibrium and with negligible gas pressure, the magnetic field lines of coronal loops as observed in the solar corona constitute geodesic foliations. In contrast to the twisted case (right), the untwisted case (left) additionally realize the Beckmann optimal transportation plan from the source flux density to sink flux density on the solar surface.}
    \label{fig:SolarLoops}
\end{figure}

Our \cref{thm:MainTheorem1,thm:MainTheorem2} allow precise characterizations of the distinction between active flux ropes and quiet harmonic fields in terms of geodesics and optimal transports.  The flux ropes consist of conformal geodesics connecting pairs of source and sink on the solar surface.  The relaxed harmonic fields, on the other hand, are conformal eikonal fields which not only comprise geodesics but also form source--sink pairings as the Beckmann (1-Wasserstein, earth-mover) optimal transportation plan from the source flux density to sink flux density~\cite{Santambrogio:2015:OT}. 

\begin{corollary}
    Potential-field models of the solar corona yield magnetic field lines that are conformally Beckmann optimal transportation paths between the magnetic sources and sinks on the sun's surface.  The more general force-free magnetic fields are conformally geodesic foliations whose topological connectivity between the source and sink ends is constrained (\cref{fig:SolarLoops}).
\end{corollary}

One can conversely explore non-eikonal geodesic foliations and draw analogies from the phenomena in solar flux ropes.  For example, one can connect a source and destination density by a bundle of geodesics with an overall twist.  The bundle becomes untwisted when the connectivity is the optimal transport (\cref{fig:HyperboloidTransport}). 
\begin{figure}[h]
    \centering
    \includegraphics[width = .85\textwidth]{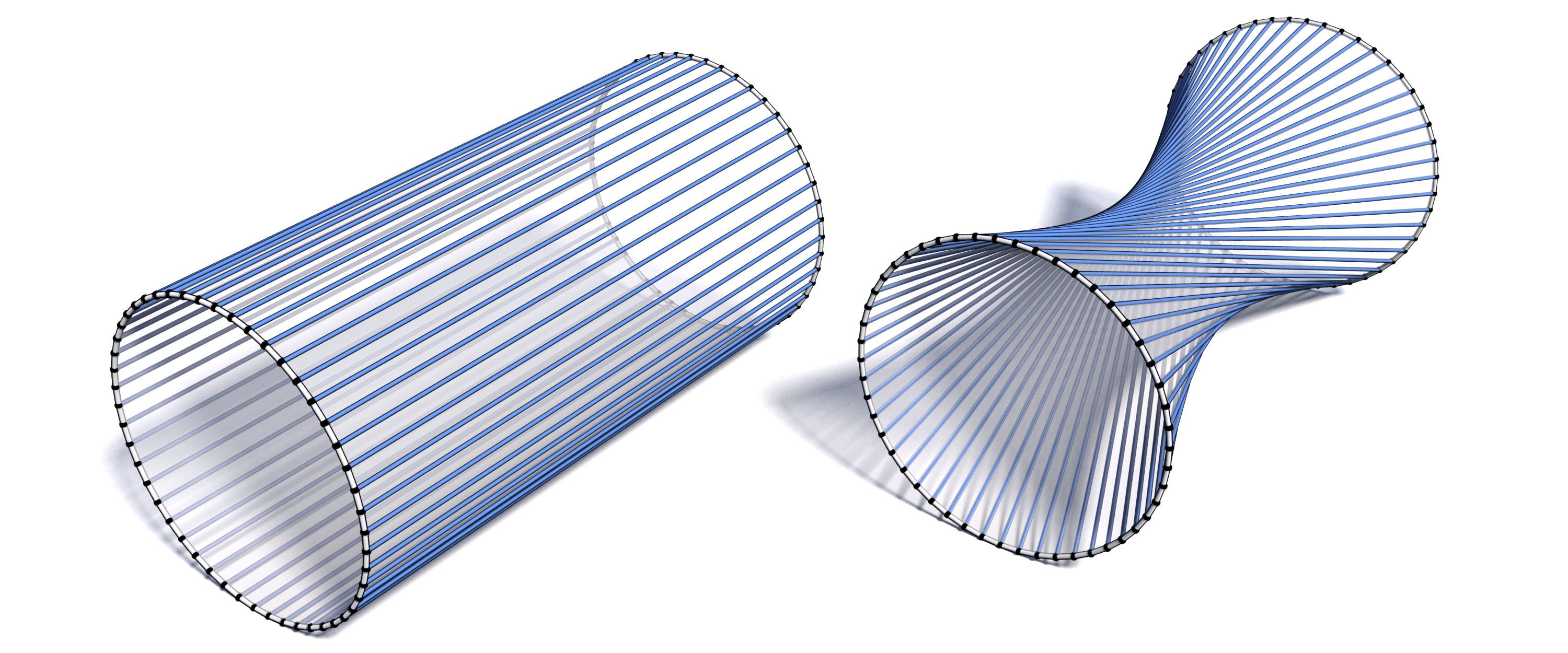}
    \caption{Left: Eikonal geodesic foliation realizing a Beckmann optimal transport plan. Right: \emph{Twisted geodesic foliation} with constrained connectivity between source and sink endpoints.}
    \label{fig:HyperboloidTransport}
\end{figure}

\section{Flux Forms in Riemannian Geometry}
\label{sec:FluxFormsInRiemannianGeometry}
Let \(M\) be a compact and oriented \(n\)-dimensional Riemannian manifold possibly with boundaries and \(\beta\in\Omega^{n-1}(M)\) a closed \((n-1)\)-form, \ie\@ \(d\beta=0\), which satisfies \(j^\ast_{\partial M}\beta = \beta_{\partial M}\) for a given boudnary condition \(\beta_{\partial M}\in\Omega^{n-1}(\partial M)\). We will refer to \(\beta\) as a \emph{flux form} and denote the Riemannian metric of \(M\) by \(g\) as well as the induced volume form, \emph{Hodge star} and norm by \(\mu\), \(\star\) and \(|\cdot|\) respectively. The Riemannian structure induces a norm on \(k\)-forms which, for \(\omega\in\Lambda^kT_p^\ast(M)\), is defined by
\begin{align}
    |\omega|^2\coloneqq \star(\omega\wedge(\star\omega)).
\end{align}
Moreover, from the non-degenerate pairing \[\langle\cdot|\cdot\rangle\colon \Lambda^kT_p^\ast(M)\times\Lambda^{(n-k)}T_p^\ast(M)\mapsto\RR,\, (\eta,\omega)\mapsto\star(\eta\wedge\omega)\] we have an isomorphism
\(\Lambda^kT_p^\ast(M)\cong\Lambda^{(n-k)}T_p^\ast(M)\).

A flux form together with a metric give rise to a \emph{vector field} \(B\in\Gamma TM\) \emph{associated to the flux form}  which is defined by 
\begin{align}
    \label{eq:AssocitatedVectorField}
    \iota_B\mu=\beta,
\end{align}
 where \(\iota\) denotes the \emph{interior product}.

\subsection{Force-Free and (Exact) Harmonic Flux Forms}
Our investigations focus on fields whose Lorentz-force \((\curl B)\times B\) vanishes. With help of the vector calculus identity \((\curl B)\times B = \nabla_BB - \tfrac{1}{2}\grad |B|^2\) we can free ourselves from the dimensional restrictions on the \(\curl\)-operator and the cross product and express this property in arbitrary dimensions. Moreover, physical forces are favorably expressed as \(1\)-forms~\cite{mackay_2020}, which suggests that we are interested in fields lines for which the \(1\)-form
\[(\nabla_BB)^\flat - \tfrac{1}{2}d|B|^2 \in\Omega^1(M)\]
vanishes. Here, \((\cdot)^\flat\) denotes the \emph{musical isomorphism} which turns a vector field \(X\in\Gamma TM\) into a \(1\)-form \(X^\flat(\cdot) = g(X,\cdot)\in\Omega^1(M)\). In the \(3\)-dimensional case, it is interpreted as the Lorentz-force, while for higher dimensions this physical picture is no longer valid. Nonetheless we will see that the corresponding fields are always co-linear with their \(\curl\) (whenever  this is reasonably defined) and therefore indeed capture an essential property of these special fields. 

\begin{lemma}\label{thm:WadsleyLemma}
	A vector field \(X\in\Gamma TM\) on a Riemannian manifold \(M\) satisfies 
	\begin{align}\label{eq:WadsleyFormula}
		\iota_XdX^\flat = (\nabla_XX)^\flat - \tfrac{1}{2}d|X|^2.
	\end{align}
\end{lemma}

\begin{proof}
Denoting the identity vector-valued 1-form \(I\in \Omega^1(M;TM)\), \(I(X)\coloneqq X\), and by the torsion-freeness \(d^\nabla I = 0\) of the connection \(\nabla\) we have \(dX^\flat = g(\nabla X\wedge I)\). Therefore, contracting with \(X\), we find
\[\iota_XdX^\flat = g(\nabla_XX,I) - g(\nabla X, X) = (\nabla_XX)^\flat -\tfrac{1}{2}d|X|^2,\]
as claimed. 
\end{proof}
\cref{thm:WadsleyLemma} suggests that the Lorentz-force of a vector field can be expressed as \(\iota_BdB^\flat\), which gives a more concise form of the expression. In particular, with 
\begin{align*}
    \beta = \star B^\flat, \qquad \star\beta = (-1)^{n-1}B^\flat
\end{align*}
it allows us to define a notion of force-free flux forms on manifolds of arbitrary dimensions.
\begin{definition}
\label{def:Beltrami}
    A closed flux form \(\beta\in\Omega^{n-1}(M)\) is called \emph{force-free} if is satisfies 
    \begin{align}
        \label{eq:BeltramiEq}
        \iota_Bd\star\beta=0.
    \end{align}
\end{definition}
As \(\curl B=0\) implies \((\curl B)\times B=0\), harmonic fields constitute an important special case of force-free fields. 
\begin{figure}[H]
    \centering
    \includegraphics[width = .95\textwidth]{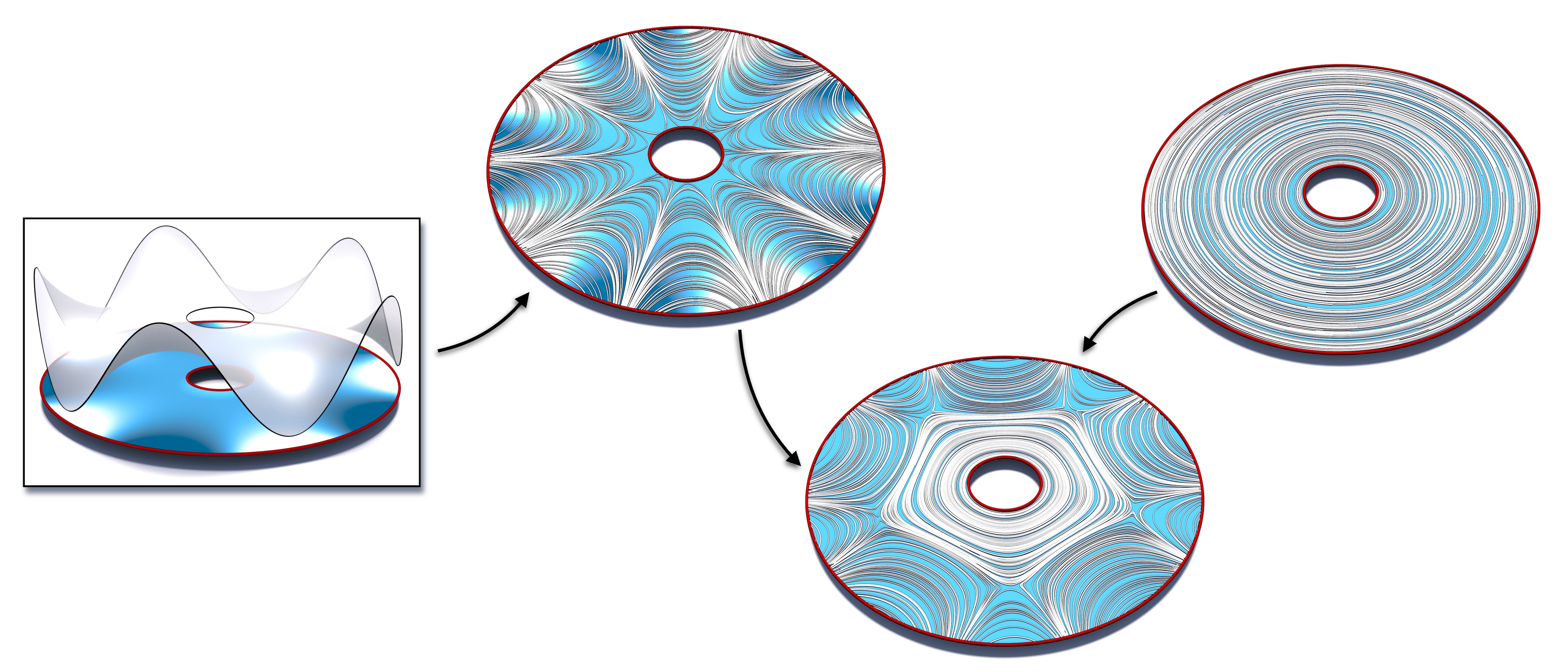}
    \caption{A non-exact harmonic vector field (bottom) build from the gradient vector field (top left) of a harmonic function (inset) and a vector field corresponding to a generator of the de Rham cohomology of the annulus (top right).}
    \label{fig:HarmonicOnAnnulus}
\end{figure}

\begin{definition}
    \label{def:(Exact)Harmonic}
    Let \(\beta\in\Omega^{n-1}(M)\) be a closed flux form. Then, 
    \begin{enumerate}
        \item \(\beta\) is called \emph{harmonic} if it is co-closed, \ie, \(d\star\beta=0\,.\)
        \item \(\beta\) is called \emph{exact harmonic} if it is co-exact, \ie, \(\beta\in\im (\star d)\,.\)
    \end{enumerate}
\end{definition}

Note that with these definitions (exact) harmonic flux forms indeed are special cases of force-free flux forms. Moreover, all exact harmonic forms are harmonic, whereas the converse does not hold. In the case that \(\beta\) is exact harmonic, the associated vector field is the gradient of some harmonic function, whereas the vector field associated to a merely harmonic flux form may have components corresponding to the non-trivial generators of the de Rham cohomology of the domain (\cref{fig:HarmonicOnAnnulus}). As an immediate consequence of \cref{eq:BeltramiEq} and \cref{def:(Exact)Harmonic} we conclude
\begin{proposition}
\label{prop:HarmoicImpliesBeltrami}
A closed flux form \(\beta\in\Omega^{n-1}(M)\) which is harmonic is force-free.
\end{proposition}

\subsubsection{Force-Free vs. Beltrami Forms}
In the realm of fluid dynamics, force-free fields on a 3-dimensional Riemannian manifold \(M\) are referred to as \emph{Beltrami fields}. They are commonly characterized as those vector fields whose \(\curl\) is co-linear to the original field, \ie, \(\curl B=\lambda B\) for some smooth function \(\lambda\in C^\infty(M)\). Generalizing these fields to dimensions \(n>3\), a common approach is to use this property as the defining property (see, \eg, \cite{Rechtman2009UAD, cardona2021geometry}). 
\begin{definition}\label{def:BeltramiField}
	Let \(M\) be a Riemannian manifold of odd dimension \(2n+1\). Then a vector field \(B\in\Gamma TM\) is Beltrami if there is \(\lambda\in C^\infty(M)\) such that
	\begin{align}
		\label{eq:BeltramiEqColinear}
		\curl B=\lambda B,
	\end{align}
	where the vector field \(\curl B\in\Gamma TM\) is defined by
	\begin{align}\label{eq:DefCurl}
		\iota_{\curl B}\mu=(dB^\flat)^n\in\Omega^{2n}(M).
	\end{align}
 The function \(\lambda\) is referred to as the \emph{proportionality factor}.
\end{definition}

Being restricted to odd-dimensional manifolds, this approach is clearly conceptually very different from \cref{def:Beltrami}. Nonetheless, the two definitions coincide on a \(3\)-dimensional Riemannian manifold. For odd-dimensions dimensions \(n>3\) there is a subtle difference, which is why in this paper we carefully distinguish between the two notions of force-free and Beltrami fields. The following \cref{thm:ForceFreeImpliesBeltrami} asserts that force-free fields are Beltrami.
\begin{proposition}\label{thm:ForceFreeImpliesBeltrami}
		Let \(B\in\Gamma TM\) be  a nowhere vanishing, divergence-free and force-free vector field on a Riemannian manifold \(M\) of odd dimension \(2n+1\). Then \(B\) is Beltrami. 
\end{proposition}
\begin{proof}
	By assumption \(B\) is force-free, \ie, \(\iota_BdB^\flat=0\). Thus, \[\iota_B\iota_{\curl B}\mu = \iota_B(dB^\flat)^n=0,\] from which we conclude the existence of a function \(\lambda\in C^\infty(M)\) such that \(\curl B=\lambda B\). 
\end{proof}
However, the converse statement only holds with an additional assumption.
\begin{definition}
	The \emph{rank} of a \(2\)-form \(\omega\in\Omega^2(M)\) is the largest power \(r\in\ZZ_{\geq1}\) such that \(\omega^r\neq 0\) and \(\omega^{r+1}=0\).
	Here, for \(p\in\ZZ_{\geq 1}\), the term \(\omega^p\) denotes the \(p\)-fold wedge product \(\omega\wedge\ldots\wedge\omega\) of \(\omega\) with itself.
\end{definition}
The rank is said to be \emph{maximal} if \(r=n\) on a manifold of even dimension \(2n\), \resp\@ odd dimension \(2n+1\). We will leave it to the reader (see, \eg,~\cite{Gross:2024:CGIMHD}) to verify 
\begin{lemma}\label{thm:GenericImplicationLemma}
	Let \(M\) be a manifold of odd dimension \(2n+1\) and \(\omega\in\Omega^2(M)\) of maximal rank. Then for every vector field \(X\in\Gamma TM\) we have that \(\iota_X\omega^n=0\) if and only if \(\iota_X\omega=0\).
\end{lemma}
Following \cite{cardona2021geometry} we will refer to the vector field \(B\in\Gamma TM\) (\resp\@ \(B^\flat\)) on a Riemannian manifold \(M\) as \emph{generic} if \(dB^\flat\) has maximal rank almost everywhere.
\begin{proposition}
		Let \(B\in\Gamma TM\) be a nowhere vanishing and generic Beltrami vector field on a Riemannian manifold \(M\) of odd dimension \(2n+1\). Then \(B\) is force-free. 
\end{proposition}
\begin{proof}
	Let \(\lambda\in C^\infty(M)\) such that \(\curl B=\lambda B\), then 
	\[\lambda\beta=\iota_{\curl B}\mu = (dB^\flat)^n\]
	and therefore 
	\[0=\lambda\iota_B\beta =\iota_B(dB^\flat)^n.\]
	By the genericity assumption, \(\lambda\) is non-vanishing almost everywhere and by \cref{thm:GenericImplicationLemma} \(B\in\ker (dB^\flat)^n\) implies \(B\in\ker dB^\flat\) almost everywhere, which yields the claim by continuity. 
\end{proof}
Unfortunately, a known equivalence between geodesible vector fields and Beltrami fields does not generalize to dimensions \(2n+1>3\). A volume preserving Beltrami field which is not geodesible is constructed in~\cite[Sec. 2.2.2]{cardona2021geometry}. 
However, in return for our slightly stronger assumptions, our \cref{def:Beltrami} preserves this equivalence, not only odd, but arbitrary dimensions. Moreover, our definition preserves the property that the defining equations for force-free forms contain (exact) harmonic forms as special cases. Lastly, again in agreement with the 3-dimensional theory, our defining equations emerge as the Euler-Lagrange equations of corresponding variational principles (\cref{sec:HierarchyOfL2,sec:HierarchyOfL1}).

\subsection{Geodesic Flux Forms}
\begin{definition}\label{def:GeodesicVF}
    A flux form \(\beta\in\Omega^{n-1}(M)\) is called \emph{geodesic} if the acceleration of its associated vector field is always proportional to itself, \ie, there is a \(\rho\in C^\infty(M)\) such that 
    \begin{align}
        \label{eq:GeodesicVectorField}    
        \nabla_BB=\rho B,
    \end{align}
    where \(\nabla\) denotes the Levi-Civita connection of the Riemannian metric \(g\). If \(\rho\neq 0\), the vector field \(B\) is called \emph{pre-geodesic}, while for  \(\rho=0\), \(B\) is called \emph{geodesic}.
\end{definition}

The field line associated to a a geodesic flux form trace out geodesics (possibly up to reparametrization,) in the Riemannian manifold. Whenever \(B\) is non vanishing we may consider the \emph{directional vector field} \(H\coloneqq|B|^{-1}B\in\Gamma TM\). The corresponding \emph{directional covector field} is given by 
\(H^\flat\).

For flux forms \(\beta\) with constant length (\ie, the associated vector fields have constant length) \cref{thm:WadsleyLemma} implies
\[0=(\nabla_BB)^\flat=\iota_BdB^\flat = \iota_Bd\star\beta,\]
from which we conclude
\begin{lemma}
\label{thm:GeodesicIff}
A flux form \(\beta\in\Omega^{n-1}(M)\) is geodesic if and only if on its support
\begin{align}
    \label{eq:EquivalentGeodesicDef}
    0=\iota_Bd\left(\tfrac{\star\beta}{|\star\beta|}\right),
\end{align}
where \(B\) is the vector field associated to \(\beta\).
\end{lemma}
Note that from \teqref{eq:EquivalentGeodesicDef} alone we can conclude that a vector field \(B\) is geodesic if and only if its directional vector field is force-free.

\subsection{Normalizations and Eikonal Flux Forms}
\label{sec:NormalizationsAndEikonalFluxForms}
The statement of \cref{thm:GeodesicIff} can be reformulated to eliminate the restriction to the support of \(\beta\). To this end, we address the ill-posedness of normalization when  \(\star\beta\) becomes zero.
\begin{definition}
    \label{def:Normalization}
    Let \(\alpha\in\Omega^k(M)\). Then a \(k\)-form \(\xi\in\Omega^k(M)\) is called a normalization of \(\alpha\in\Omega^k(M)\) if 
    \[|\alpha|\xi=\alpha \quad\text{and}\quad |\xi|\leq 1.\]
\end{definition}
At every point \(p\in M\), a normalization can be seen as an element of the subdifferential \(\partial|\beta|\) (cf. \cref{sec:HierarchyOfL1}). Thus, whenever the flux form \(\beta\) is non-vanishing, it is uniquely determined. In particular, on the support of a flux form, a normalization coincides with the directional covector field. Therefore, we may more adequately state \cref{thm:GeodesicIff} as follows: 
\begin{proposition}\label{thm:GeodesicIffNormalizationProp}
A closed flux form \(\beta\in\Omega^{n-1}(M)\) is geodesic if and only if there exists a normalization \(\eta\in\Omega^{1}(M)\) of \(\star\beta\) such that \[0=\iota_Bd\eta.\] 
\end{proposition}

As pointed out in \secref{sec:Introduction}, \emph{eikonal fields} are a special subclass of geodesic vector fields. They describe gradients of distance functions and therefore have a unit norm. Thus, the corresponding covector fields are closed normalizations of the corresponding flux forms.
\begin{definition}
    \label{def:Eikonal}
    A closed flux form \(\beta\in\Omega^{n-1}(M)\) is called \emph{eikonal} (\resp\@ \emph{exact eikonal}) if there exists a closed (\resp\@ exact) 
    normalization  \(\eta\in\Omega^1(M)\) of \(\star\beta\).
\end{definition}
\begin{proposition}
\label{prop:EikonalImpliesGeodesic}
A closed flux form \(\beta\in\Omega^{n-1}(M)\) which is eikonal is geodesic.
\end{proposition}

\subsection{Flux Forms in Conformal Geometry}
A \emph{conformal class} on an \(n\)-dimensional smooth manifold \(M\) is an equivalence class of Riemannian metrics, where two metrics are  \(g\) and \(h\) are considered \emph{conformally equivalent} if there exists a smooth function \(u\in C^\infty(M)\) such that 
\begin{align}
    \label{eq:ConformalMetric}
    e^{2u}g= h.
\end{align}
A manifold \(M\) together with a conformal structure (denoted by \([g]\)) is referred to as \emph{conformal manifold}.

While the quantities associated with the flux form, as defined in \cref{sec:FluxFormsInRiemannianGeometry}, rely on the specific choice of a Riemannian metric, the flux form \(\beta\) itself and, consequently, the geometry of the corresponding field lines are independent of the metric. Hence, it is possible to define special types of flux forms on a conformal manifold by requiring the existence of a representative metric within the equivalence class that satisfies the defining equations. Consequently, we introduce the following definitions:
\begin{definition}
\label{def:ConformallyVerions}
A closed flux form \(\beta\in\Omega^{n-1}(M)\) on a conformal manifold \(M\) is called 
\begin{enumerate}
    \item \emph{conformally force-free} if there exists a metric in the conformal class of \(M\) such that \(\beta\) is force-free.
    \item \emph{conformally geodesic} if there exists a metric in the conformal class of \(M\) such that \(\beta\) is geodesic.
    \item \emph{conformally harmonic} (\resp\@ \emph{conformally exact harmonic}) if there exists a metric in the conformal class of \(M\) such that \(\beta\) is  harmonic (\resp\@ exact harmonic).
    \item \emph{conformally eikonal} (\resp\@ \emph{conformally exact eikonal}) if there exists a metric in the conformal class of \(M\) such that \(\beta\) is eikonal (\resp\@ exact eikonal).
\end{enumerate}
\end{definition}

Note that with the metric independence of \(\beta\), also the statements of \cref{prop:HarmoicImpliesBeltrami} and \cref{prop:EikonalImpliesGeodesic} carry over to the conformal setup. 

\section{A Hierarchy of Variational Principles for the \(L^2\)-Norm}
\label{sec:HierarchyOfL2}
Both, force-free and the more specialized cases of (exact) harmonic flux forms can equivalently be characterized in terms of variational principles. To this end, one considers the \(L^2\)-norm of the flux form, which is given by 
\begin{align}
    \label{eq:L2Norm}
    \|\beta\|^2_2 \coloneqq \int_M\beta\wedge\star\beta\,.
\end{align}
The different cases then emerge as stationary points of the \(L^2\)-norm under different classes of variations with suitable boundary conditions. 
\begin{theorem}
    \label{thm:VariationalExactHarmonic}
    A closed flux form \(\beta\in\Omega^{n-1}(M)\) with \(j^\ast_{\partial M}\beta=\beta_{\partial M}\) for given boundary data \(\beta_{\partial M}\in\Omega^{n-1}(M)\) is a stationary point of the \(L^2\)-norm (\(d\mathring\beta = 0\) and \(j^\ast_{\partial M}\mathring\beta = 0\)) if and only if \(\beta\) is exact harmonic. 
\end{theorem}

\begin{proof}
    The vanishing variation condition of the \(L^2\)-norm \eqref{eq:L2Norm} is given by \[0=\int_M\mathring\beta\wedge\star\beta\] for all \(\mathring\beta\) satisfying \(d\mathring\beta = 0\) and \(j^*_{\partial M}\mathring\beta = 0\).  That is, the stationary condition is equivalent to \[\beta\in\{\mathring\beta\in \Omega^{n-1}(M)\,\vert\, d\mathring\beta = 0,j_{\partial M}^*\mathring\beta = 0\}^\bot = \im(\star d)\] where the last equality is given by the Hodge--Morrey--Friedrichs decomposition~\cite{schwarz2006hodge}.
\end{proof}

\begin{theorem}
    \label{thm:VariationalHarmonic}
    A closed flux form \(\beta\in\Omega^{n-1}(M)\) with \(j^\ast_{\partial M}\beta=\beta_{\partial M}\) for given boundary data \(\beta_{\partial M}\in\Omega^{n-1}(M)\) is a stationary point of the \(L^2\)-norm under \emph{homologically constrained variations}, \ie\@ \(\mathring\beta = d\alpha\) for some \(\alpha\in\Omega^{n-2}(M)\) with \(j^\ast_{\partial M}\alpha = 0\),  if and only if \(\beta\) is harmonic. 
\end{theorem}

\begin{proof}
    The vanishing variation condition of the \(L^2\)-norm \eqref{eq:L2Norm} under variations \(\mathring\beta = d\alpha\), \(j_{\partial M}^*\alpha\), is given by \[0=\int_M d\alpha\wedge\star\beta = (-1)^{n-1}\int_M\alpha\wedge d\star\beta\] for all \(\alpha\in\Omega^{n-2}(M)\) with \(j_{\partial M}^*\alpha = 0\).  This condition holds if and only if \(d\star\beta = 0\). 
\end{proof}

\begin{theorem}
    \label{thm:VariationalBeltrami}
    A closed flux form \(\beta\in\Omega^{n-1}(M)\) with \(j^\ast_{\partial M}\beta=\beta_{\partial M}\) for given boundary data \(\beta_{\partial M}\in\Omega^{n-1}(M)\) is a stationary point of the \(L^2\)-norm under \emph{isotopy constraint variations}, \ie\@ \(\mathring\beta = -\LD_\xi\beta \) for some \(\xi\in \Gamma TM\) which is compactly supported in the interior of \(M\),  if and only if \(\beta\) is force-free. 
\end{theorem}

\begin{proof}
    By Cartan's formula and \(d\beta = 0\), the isotopic variations take the form \(\mathring\beta = -\LD_\xi\beta = -d\ip_{\xi}\beta\) for compactly supported vector fields \(\xi\in\Gamma TM\).
    The variation of \eqref{eq:L2Norm} under such variation is given by
    \begin{align*}
    \tfrac{1}{2}(\|\beta\|_2^2\mathring) &= 
    \int_M - d\ip_\xi\beta\wedge\star\beta = (-1)^n\int_M \ip_\xi\beta\wedge d\star\beta
    =\int_M\star(\xi^\flat\wedge\beta)\wedge d\star\beta \\ &= 
    \int_M\xi^\flat\wedge(\star\beta)\wedge\star d\star\beta=
    (-1)^n\int_M\xi^\flat\wedge\star (\ip_B d\star\beta).
    \end{align*}
    Therefore, the vanishing variation condition for all compactly supported \(\xi\in\Gamma TM\) is equivalent to \(\ip_B d\star\beta = 0\), \ie\@ \(\beta\) is force-free.
\end{proof}

\section{A Hierarchy of Variational Principles for the \(L^1\)-Norm}
\label{sec:HierarchyOfL1}
We now derive the stationary conditions of \(L^1\)-optimization problems with the same sets of boundary conditions and constraints on the variations we have employed for the \(L^2\)-case.  To this end we first note that the integrand \(|\beta|\) of the \(L^1\)-norm fails to be smooth at vanishing points of \(\beta\). Therefore, when considering variations of the \(L^1\)-norm
\begin{align}
        \label{eq:L1EL}
        \biggl(\int_M|B|\,\mu\mathring{\biggr)} = \int_M\biggl( \sqrt{\star\left( \beta\wedge\star\beta\right)} \mathring{\biggr)} = \int_M \mathring{\beta}\wedge \partial|\beta|\,, %\tfrac{\star\beta}{|\star\beta|} \,,
    \end{align} 
we need to resort to the subdifferential
\begin{align}
    \partial|\beta|=\begin{cases}
        \tfrac{\star\beta}{|\star\beta|} & \text{if}\ \beta\neq 0 \\
        \{\alpha\in\Omega^1(M)\mid |\alpha|\leq 1\} & \text{if}\ \beta= 0
    \end{cases}
\end{align}
of \(\beta\) in order to state the stationary conditions. As pointed out in \cref{sec:NormalizationsAndEikonalFluxForms}, the subdifferential \(\partial|\beta|\) consists of the normalizations \(\xi\) of \(\star\beta\). 
\begin{lemma}
    \(\partial|\beta|= \{\xi\in\Omega^1(M)\mid |\xi|\leq1,\ |\star\beta|\xi=\star\beta \}\) 
\end{lemma}
\begin{proof}
Let \(\eta\in\partial |\beta|\). When \(\beta\neq0\), then \(\eta=\tfrac{\star\beta}{|\star\beta|}\) and therefore \(|\eta|=1\). Moreover, when \(\beta=0\), then \(\eta\in\Omega^1(M)\) which (by definition) satisfies \(|\eta|\leq 1\). Clearly, also \(0\cdot\eta=0\) and therefore \(\eta\) is a normalization. 

Let conversely \(\eta\in\Omega^1(M)\) be a normalization of \(\star\beta\), \ie\@ \(|\star\beta|\eta=\star\beta\) and \(|\eta|\leq 1\). By definition, the subdifferential of \(|\beta|\) is given by 
\[\partial|\beta|=\{\alpha\in\Omega^1(M)\mid |\tilde\beta|\geq|\beta| + \langle\alpha\,|\,\tilde\beta-\beta\rangle\quad \forall\tilde\beta\in\Omega^{n-1}(M)\}.\] Now if \(\beta=0\), then \[|\eta|\leq 1 \quad\Leftrightarrow \quad\sup_{\tilde\beta\in\Omega^{n-k}(M),\, |\tilde\beta|=1}\langle\eta\,|\,\tilde\beta\rangle \leq 1 \quad\Leftrightarrow\quad\langle\eta\,|\,\tilde\beta\rangle \leq |\tilde\beta|\quad \forall\tilde\beta\in\Omega^{n-1}(M).\]
Moreover, if \(\beta\neq0\) we have that \(\langle\eta\,|\,\beta\rangle = |\star\beta|=|\beta|\) and hence 
\[|\tilde\beta|\geq|\beta| + \langle\eta\,|\,\tilde\beta-\beta\rangle = \langle\eta\,|\,\tilde\beta\rangle \quad \forall\tilde\beta\in\Omega^{n-1}(M)\}\] holds if an only if \(|\eta|\leq 1\), which is true by assumption.
\end{proof}

\begin{theorem}
    \label{thm:VariationalExactEikonal}
    A closed flux form \(\beta\in\Omega^{n-1}(M)\) with \(j^\ast_{\partial M}\beta=\beta_{\partial M}\) for given boundary conditions \(\beta_{\partial M}\in\Omega^{n-1}(M)\) is a stationary point of the \(L^1\)-norm (\(d\mathring\beta = 0\) and \(j^\ast_{\partial M}\mathring\beta = 0\)) if and only if \(\beta\) is exact eikonal. 
\end{theorem}

\begin{proof}
    Analogous to the proof of \cref{thm:VariationalExactHarmonic} we conclude from \cref{eq:L1EL} that the stationary condition \[0\in \int_M \mathring{\beta}\wedge \partial|\beta|\quad \text{for all \(\mathring\beta\) with \(d\mathring\beta = 0\) and \(j^\ast_{\partial M}\mathring\beta = 0\)}\] 
    is equivalent to the existence of an exact normalization \(\eta\in\partial|\beta|\) of \(\star\beta\), \ie\@ \(\beta\) is exact eikonal (\cref{def:Eikonal}).
\end{proof}

\begin{theorem}
    \label{thm:VariationalEikonal}
    A closed flux form \(\beta\in\Omega^{n-1}(M)\) with \(j^\ast_{\partial M}\beta=\beta_{\partial M}\) for given boundary conditions \(\beta_{\partial M}\in\Omega^{n-1}(M)\) is a stationary point of the \(L^1\)-norm under \emph{homologically constraint variations}, \ie\@ \(\mathring\beta = d\alpha\) for some \(\alpha\in\Omega^{n-2}(M)\) with \(j^\ast_{\partial M}\alpha = 0\),  if and only if \(\beta\) is eikonal. 
\end{theorem}

\begin{proof}
    The stationary condition for the variation of the \(L^1\)-norm under variations \(\mathring\beta = d\alpha\), \(j^\ast_{\partial M}\alpha = 0\) is given by 
    \begin{align*}
    0\in \int_M \mathring{\beta}\wedge \partial|\beta| = (-1)^{n-1}\int_M\alpha\wedge d(\partial|\beta|)
    \end{align*}for all \(\alpha\in\Omega^{n-2}(M)\) with \(j_{\partial M}^*\alpha = 0\), which is equivalent to the existence of a closed normalization \(\eta\in\partial|\beta|\) of \(\star\beta\), \ie\@ \(\beta\) is eikonal (\cref{def:Eikonal}).
\end{proof}

\begin{theorem}
    \label{thm:VariationalTwistedGeodesic}
    A closed flux form \(\beta\in\Omega^{n-1}(M)\) with \(j^\ast_{\partial M}\beta=\beta_{\partial M}\) for given boundary conditions \(\beta_{\partial M}\in\Omega^{n-1}(M)\) is a stationary point of the \(L^1\)-norm under \emph{isotopy constraint variations}, \ie\@ \(\mathring\beta = -\LD_\xi\beta \) for some \(\xi\in \Gamma TM\) which is compactly supported in the interior of \(M\),  if and only if there exists a normalization \(\eta\in\partial|\beta|\) of \(\star\beta\) such that 
\((\star\beta)\wedge(\star d\eta)=0\). 
\end{theorem}

\begin{proof}
With analogous arguments as for \cref{thm:VariationalBeltrami} the vanishing condition for all compact-support \(\xi\in\Gamma TM\) is given by
\begin{align*}
0\in -\int_M\partial|\beta|\wedge d\iota_\xi\beta = -\int_M \xi^\flat\wedge\left( (\star\beta)\wedge (\star d (\partial|\beta|))\right),
\end{align*}
which is equivalent to \(0\in (\star\beta)\wedge (\star d (\partial|\beta|))\), \ie\@ the existence of a normalization \(\eta\in\partial|\beta|\) of \(\star\beta\) which satisfies \(0=(\star\beta)\wedge (\star d \eta)\).
\end{proof}
 
The Karush–Kuhn–Tucker (KKT) condition~\cite{Boyd_Vandenberghe_2004} \((\star\beta)\wedge(\star d\eta)=0\) can equivalently be expressed in terms of the associated vector field \(B\) as 
\begin{align}
    \iota_Bd\eta=0
\end{align}
Note that on the support of \(\beta\), the normalization agrees with the directional covector field (\cref{thm:GeodesicIffNormalizationProp}), hence the KKT-condition suggests that the field lines form a geodesic foliation. We refer to these fields as \emph{twisted geodesic foliations} as they do not necessarily solve an optimal transport problem. The corresponding \emph{untwisted} cases solve a Beckmann optimal transport problem and correspond to (exact) eikonal fields (\cref{fig:HyperboloidTransport}).

\begin{remark}[Twisted Minimal Foliations]
    In the field of calibrated geometry~\cite{Harvey:1982:CG, Zhang2013Gluing}, the directional covector field \(\eta\) is referred to as a \emph{calibration}. In more generality, a calibration is a closed form \(\alpha\in\Omega^{k}(M)\) which, for every oriented \(k\)-dimensional subspace \(V\subset T_pM\), satisfies \(\alpha\vert_V\leq \mu_V\), where \(\mu_V\) is the volume form on \(V\) induced by the Riemannian metric. The existence of a calibration gives rise to a foliation of minimal \(k\)-dimensional submanifolds---in our setup a geodesic foliation by field lines. On the basis of the hierarchy of KKT-conditions 
    \begin{equation}
        \{\eta = d\alpha\}\quad\subset\quad \{d\eta=0\} \quad\subset\quad \{\ip_Bd\eta=0\}
    \end{equation}
    we have introduced in this section it is an interesting endeavor to investigate \emph{twisted minimal foliations}, generalizations of twisted geodesic foliations for calibrations with \(k\geq 2\).
\end{remark}

\section{Conformal Change of Metric}
The problem of minimizing the \(L^2\)-norm of a magnetic field in \(\RR^3\) can be approached by introducing a conformal change of the form \(|B|^2 g\), for a non-vanishing magnetic field \(B\)~\cite{Zhang2013Gluing, Padilla22filament,Gross:2023:PK}. This particular (\(B\)-dependent) conformal factor has interesting consequences and explicitly ties together seemingly unrelated fields. More specifically, it turns out that the KKT-conditions for the \(L^1\)-optimization problems can equivalently be derived from the Euler-Lagrange equations for the \(L^2\)-optimization problems by applying a conformal change of metric.

Consider a closed flux form \(\beta\) and a representative of the conformal class \(\bowg\in[\bowg]\). From these given objects, we may construct a conformally changed metric \(\barg\in[\bowg]\) on the support of \(\beta\) by defining
\begin{align}
    \label{eq:ConformalChange}
    \barg\coloneqq|\beta|^2_{\bowg}\,\bowg.
\end{align}
This conformal change determines transformation rules for all metric dependent objects which we defined in \cref{sec:FluxFormsInRiemannianGeometry}: denoting the volume forms induced by the respective metrics by \(\bowmu\) \resp\@ \(\barmu\), the vector fields \(\bowB,\barB\) associated to a \(\beta\) are determined by 
\begin{align}
    \label{eq:BBowBarDefinition}
    \beta = \iota_{\bowB}\bowmu = \iota_{\barB}\barmu\,.
\end{align}
For \(n\geq 3\) they can be expressed in terms of one another as
\begin{align}
    \label{eq:BTransformation}
    \bowB = |\barB|_{\barg}^{-\frac{n}{n-2}}\barB,\quad \barB = |\bowB|_{\bowg}^{-n}\bowB\,,
\end{align}
whereas the corresponding volume forms satisfy
\begin{align}
    \label{eq:MuTransformation}
    \bowmu = |\barB|_{\barg}^{\frac{n}{n-2}}\barmu,\quad \barmu = |\bowB|_{\bowg}^{n}\bowmu
\end{align}
and therefore
\begin{align}
    \label{eq:ConformalFactor}
    |\beta|_{\bowg}=|\bowB|_{\bowg}=|\barB|_{\barg}^{-\frac{1}{n-2}}\,,\quad
    |\beta|_{\barg}=|\barB|_{\barg} = |\bowB|_{\bowg}^{-(n-2)}\,.
\end{align}
Moreover we have
\begin{align}
    \label{eq:HodgeBetaTransformation}
    \bowstar\beta = |\barB|_{\barg}^{-1}\barstar\beta\,,\quad \barstar\beta = |\bowB|_{\bowg}^{-(n-2)}\bowstar\beta\,.
\end{align}

\subsection{Conformal Transformations of Stationary Conditions}
Having established the transformation rules for the individual objects in the Euler-Lagrange equations for \(n\geq 3\), we may derive the corresponding stationary conditions with respect to the conformally changed metric.

Let \(\beta\) be a closed flux form and exact harmonic with respect to \(\bowg\). Then there exists \(\phi\in C^\infty(M)\) such that \(\bowstar\beta=d\phi\) and by \cref{eq:HodgeBetaTransformation}, whenever \(\beta\) is non-zero, we have
\begin{align}
    d\phi = \bowstar\beta = |\barB|_{\barg}^{-1}\barstar\beta.
\end{align}
Globally, this can be stated by saying that there exists \(\phi\in C^\infty(M)\) such that \(d\phi\) is a normalization of \(\barstar\beta\), \ie, \(\beta\) is exact eikonal with respect to \(\barg\).

Similarly, let \(\beta\) be a closed flux form and harmonic with respect to \(\bowg\). Then \(d\bowstar \beta=0\) and by \cref{eq:HodgeBetaTransformation}, whenever \(\beta\) is non-zero, we have
\begin{align}
    0 = d\bowstar\beta = d(|\barB|_{\barg}^{-1}\barstar\beta),
\end{align}
which can be globally stated by asking for the existence of a closed normalization \(\bareta\in \Omega^1(M)\) of \(\barstar\beta\), \ie, \(\beta\) is eikonal with respect to \(\barg\) (see also~\cite{Zhang2013Gluing}).

Finally, let \(\beta\) be a closed flux form which is force-free with respect to \(\bowg\). Then \( \iota_{\bowB} d\bowstar \beta=0\) and by \cref{eq:HodgeBetaTransformation}, whenever \(\beta\) is non-zero, we have
\begin{align}
    0 = \iota_{\bowB}d\bowstar\beta = |\barB|^{-\frac{n}{n-2}}_{\barg}\,\iota_{\barB}d(|\barB|_{\barg}^{-1}\barstar\beta)\,.
\end{align}
This can be stated globally by asking for the existence of a normalization of \(\bareta\in \Omega^1(M)\) of \(\barstar\beta\) which satisfies \(0=\iota_{\barB}d\bareta\), \ie, the vector field \(\barB\) associated with \(\beta\) forms---up to reparametrization---a geodesic foliation. 

\subsection{Main Theorem}
Considering the squared \(L^2\)-norm of a flux form and apply the conformal change of metric we have
\begin{align}
    \label{eq:ChangeOfMetricInIntegral}
    \|\beta\|_{L^2,\bowg}^2 = \int_M |\bowB|_{\bowg}^2\,\bowmu = \int_M|\barB|_{\barg}\,\barmu = \|\beta\|_{L^1,\barg}. 
\end{align}
Moreover, we note that the constraints and boundary conditions in \cref{thm:VariationalExactHarmonic,thm:VariationalHarmonic,thm:VariationalBeltrami} were expressed independent of a metric. Therefore, after fixing the respective metrics, we conclude
\begin{theorem}
    \label{thm:L2ToL1}
    For \(n\geq 3\), after the conformal change of metric \(\barg = |\beta|_{\bowg}^2\,\bowg\), stationary points of the squared \(L^2\)-norm with respect to \(\bowg\) become stationary points of the \(L^1\)-norm with respect to \(\barg\) with the same constraints and boundary conditions and vice versa.
\end{theorem}

\begin{remark}[Flux-Forms with Non-Global Support]
It is well-known that stationary points of \(L^1\)-optimization problems, such as Beckmann optimal transport problems, typically exhibit sparse support (\cite{Santambrogio:2015:OT}, \cref{fig:L1vsL2}). Specifically, for points \(p\in M\) where  \(\beta\) vanishes, it is not possible to define a non-degenerate metric using \(|\beta|^2\) as a conformal factor. However, the integrity of our theory, which focuses on the geometry of field lines, remains unaffected. The concept of a field line associated with a flux form inherently assumes that the flux form is non-vanishing. Consequently, all the theory and results presented in this paper are only well-defined within the support of the flux form and whenever one of the integrals in \cref{eq:ChangeOfMetricInIntegral} is defined.
\end{remark}

Taking into account different constraints on the admissible variations we conclude:
\begin{theorem}
    \label{thm:MainTheorem1Formal}
    %Harmonic fields are conformally eikonal.
    Let \(M\) be an \(n\)-dimensional conformal manifold, \(n\geq 3\), \(\beta\in\Omega^{n-1}(M)\) be a closed flux form with \(j^\ast_{\partial M}\beta=\beta_{\partial M}\) for given boundary conditions \(\beta_{\partial M}\in\Omega^{n-1}(M)\) and \(\bowg,\barg\in[\bowg]\) be related by \(\barg = |\beta|_{\bowg}^2\,\bowg\). Then:
    \begin{enumerate}
        \item \(\beta\) is force-free with respect to \(\bowg\) if and only if it is geodesic with respect to \(\barg\).
        \item \(\beta\) is harmonic with respect to \(\bowg\) if and only if it is eikonal with respect to \(\barg\).
        \item \(\beta\) is exact harmonic with respect to \(\bowg\) if and only if it is exact eikonal with respect to \(\barg\).
    \end{enumerate}
\end{theorem}
Note that, contrasting previous work (see,\eg, \cite{cardona_2021, cardona2021geometry}), \cref{thm:MainTheorem1Formal} holds in arbitrary dimensions \(n\geq 3\) while still preserving the equivalence results. Moreover, in agreement with the 3-dimensional theory, our defining equations emerge as the Euler-Lagrange equations of variational principles and retain the known inclusions of the special cases from each other.

\begin{corollary}
\label{thm:UnitBetaBeltramiGeodesic}
If \(|\beta| = 1\), then \(\beta\) is force-free if and only if \(\beta\) is geodesic.
\end{corollary} 
\begin{figure}[h]
    \centering
    \includegraphics[width = \textwidth]{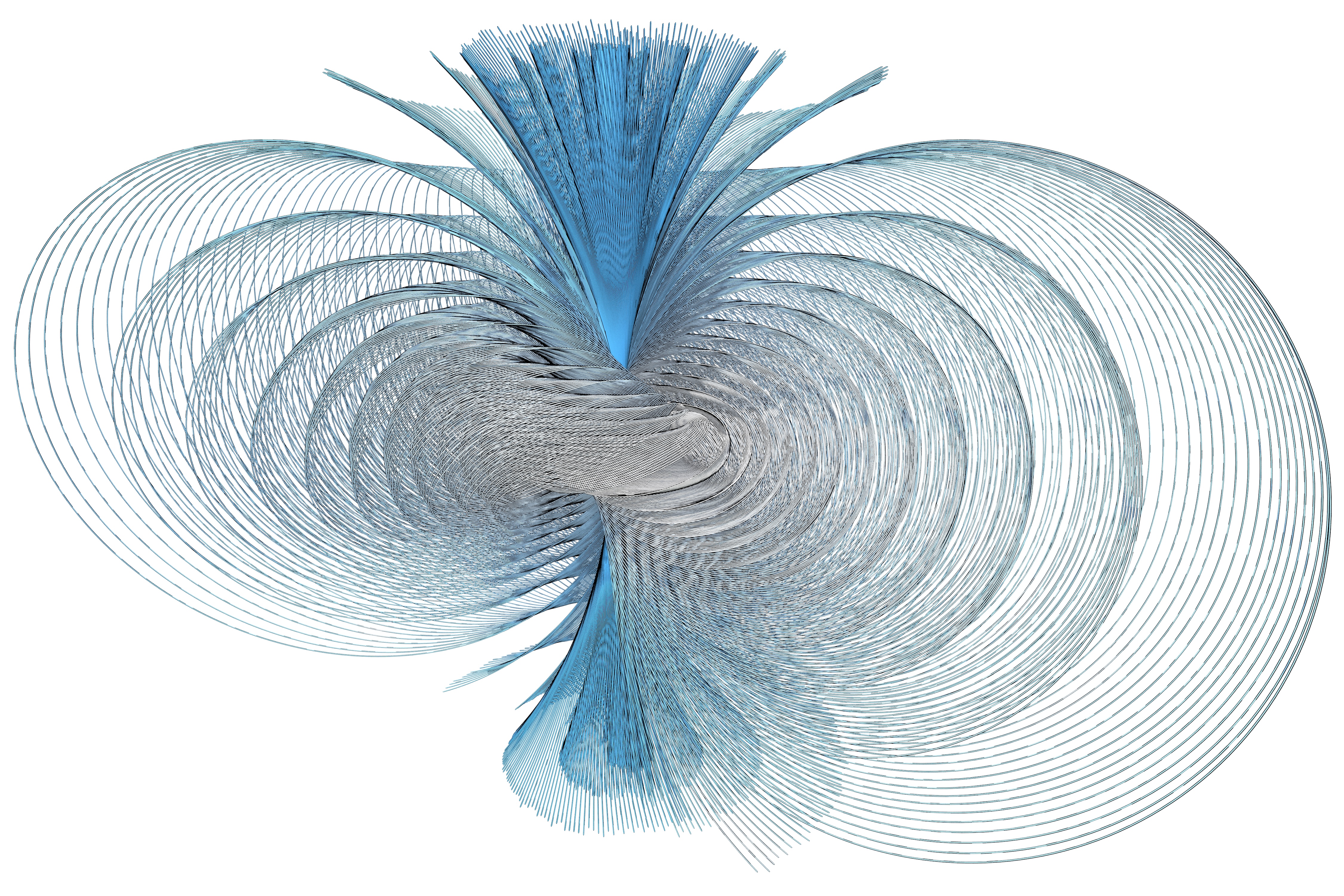}
    \caption{Depicted is the \emph{Hopf fibration}, \ie, the stereographic projection of the field lines of the \emph{Hopf field} onto \(\RR^3\). It is a geodesic and Killing vector field of unit length and therefore force-free, by \cref{thm:ConstKillingIffGeodesic} and \cref{thm:UnitBetaBeltramiGeodesic}.}
    \label{fig:HopfFibration}
\end{figure}
\begin{example}[Hopf Fibration]\label{exp:HopfFibration}
A non-trivial example for \cref{thm:UnitBetaBeltramiGeodesic} is given by the Hopf fibration (\cref{fig:HopfFibration}), which is obtained by stereographic projection of the Hopf field 
\[X_{\rm Hopf}= (-x_2,x_1,-x_4,x_3)\in \Gamma TS^3\] 
on the round \(3\)-sphere \(S^3=\{x\in\RR^4\mid x_1^2 + x_2^2 + x_3^2 + x_4^2 = 1\}\subset\RR^4\) (\figref{fig:HopfFibration}). Since the Hopf field is divergence-free, has unit length and great circles as its integral curves, by~\cref{thm:UnitBetaBeltramiGeodesic}, \(X\) is force-free (see also~\cite{Smiet2017IRHopf}). The Hopf fibration is considered, \eg, when studying so-called Hopfions in electromagnetism~\cite{PhysRevLett.111.150404} or knotted structures in ideal plasma~\cite{Smiet2017IRHopf}.
\end{example}

Previous work \cite{Etnyre:2000:CTH, rechtman_2010, cardona2021geometry} based on results on the geodesibility of vector fields by \cite{Gluck:79:OL, sullivan1978foliation} already allows to conclude an equivalence between force-free fields and geodesible vector fields in the following sense: if there is a Riemannian metric for which a vector field is force-free, then there is a metric for which the vector field is geodesic. However, the two metric are in no relation whatsoever. Our \cref{thm:MainTheorem1Formal} provides an explicit relation between the relevant metrics, thus extending the previous work. In particular, \cref{thm:UnitBetaBeltramiGeodesic} reveals for when the two metrics even coincide.

These more explicit statements are relevant from a practical point of view. For example, based on a structure-preserving discretization, \cite{Padilla22filament, Gross:2023:PK} have reduced a numerically challenging~\cite{Dixon89Agot} volumetric energy minimization with free boundary conditions to a problem of minimizing the length of curves in a conformally changed metric corresponding to our theory.

\subsection{The Surface Case}
In the case that \(M\) is a surface, \ie, \(n=2\), we find that only one implication of the equivalences in \thmref{thm:MainTheorem1Formal} holds. The reason for that is that the essential tool for the proof of \thmref{thm:L2ToL1} is the transformation of the Hodge stars under a conformal change of metric. However, the Hodge star on \(1\)-forms on a \(2\)-dimensional manifold is conformally invariant~\cite{besse2007einstein}. Therefore, for a \(2\)-dimensional manifold harmonicity (\(d\beta=0\) and \(d\star\beta=0\)) is a conformally invariant notion and cannot be achieved by a conformal transformation. 

It is easy to see that in the \(2\)-dimensional case, force-freeness and harmonicity is equivalent. This only leaves (exact) harmonic fields for our consideration. 

\begin{corollary}
	Let \(M\) be a \(2\)-dimensional conformal manifold, \(\beta\in\Omega^{n-1}(M)\) be a closed flux form with \(j^\ast_{\partial M}\beta=\beta_{\partial M}\) for given boundary conditions \(\beta_{\partial M}\in\Omega^{n-1}(M)\) and \(\bowg,\barg\in[\bowg]\) be related by \(\barg = |\beta|_{\bowg}^2\,\bowg\). Then, if \(\beta\) is (exact) harmonic with respect to \(\bowg\), \(\beta\) is (exact) eikonal with respect to \(\barg\).
\end{corollary}

\begin{proof}
	The proof is analogous to the corresponding direction to proof \thmref{thm:MainTheorem1Formal}.
\end{proof}

\begin{figure}[h]
    \centering
    \includegraphics[width = .85\textwidth]{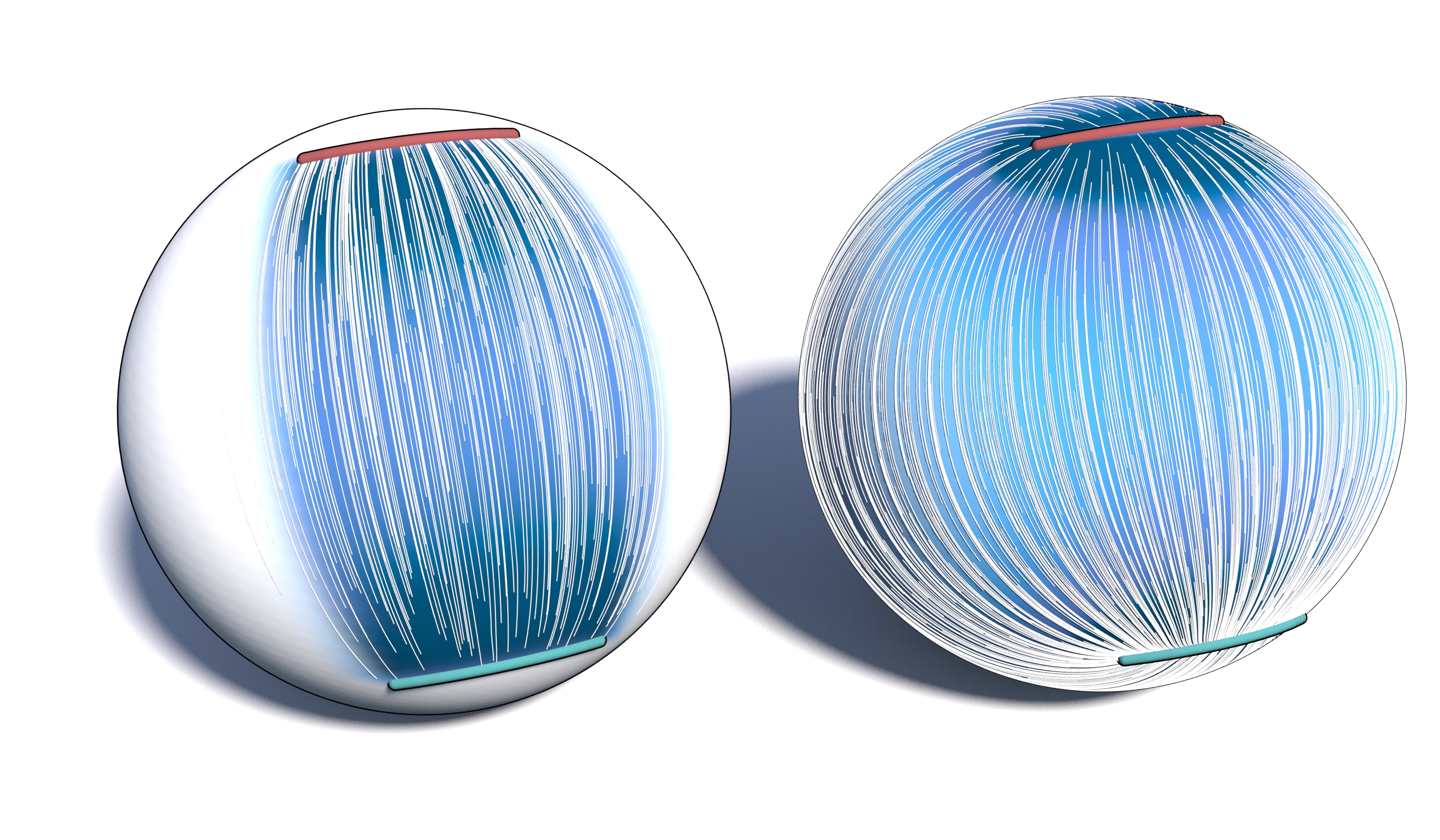}
    \caption{Vector fields \(B\in\Gamma(S^2)\) minimizing the \(L^1\)-, \resp\@ \(L^2\)-norm with boundary conditions given by a source (red) and a sink (blue).
    }
    \label{fig:L1vsL2}
\end{figure}

\begin{theorem}
	Let \(M\) be a \(2\)-dimensional conformal manifold, \(\beta\in\Omega^{n-1}(M)\) be a closed and eikonal flux form with \(j^\ast_{\partial M}\beta=\beta_{\partial M}\) for given boundary conditions \(\beta_{\partial M}\in\Omega^{n-1}(M)\). 
	Then \(\beta\) is harmonic if and only if either of the two conditions holds:
	\begin{enumerate}
		\item \(|\beta|_g\) is constant. \label{item:HarmonicityProofPartA}
		\item \(\grad |\beta|_g\) and \(B\) are parallel. \label{item:HarmonicityProofPartB}
	\end{enumerate}
	In particular, if \(\beta\) is harmonic it is harmonic with respect to any metric in \([g]\).
\end{theorem}

\begin{proof}
	By assumption \(0=d\beta=d\star B^\flat\). Thus, harmonicity of \(\beta\) is equivalent to \(dB^\flat=0\).
	Since \(\beta\) is eikonal, we have \(d(\tfrac{B^\flat}{|B|_g})=0\). Therefore,
	\[dB^\flat = (d|B|_g)\wedge(\tfrac{B^\flat}{|B|_g}).\]
	The right-hand side vanishes if and only if either \ref{item:HarmonicityProofPartA} or \ref{item:HarmonicityProofPartB} hold and since \(\star\) is conformally invariant, this is true for any metric in the conformal class.
\end{proof}

\section{Applications to Special Vector Fields}
In this section we showcase some examples where our theory provides insights about other special kinds of vector fields.
\subsection{Reeb Vector Fields}
%\begin{example}
On an orientable manifold of odd dimension \(2n+1\), \(\alpha\in\Omega^1(M)\) is said to be \emph{contact \(1\)-form} if
	\begin{equation*}
		\alpha\wedge(d\alpha)^n\neq 0.
	\end{equation*} 
 Any contact \(1\)-form describes a hyperplane distribution \(\Xi\coloneqq\ker \alpha\) and vice versa. The hyperplane distribution \(\Xi\) is referred to as a \emph{contact structure} on \(M\) and the pair \((M,\Xi)\) is a \emph{contact manifold}. Note that this relationship is not unique and any other contact \(1\)-form determining \(\Xi\) is of the form \(f\alpha\) for a non-vanishing \(f\in C^\infty(M)\). 

The standard example for a contact \(1\)-form on \(\RR^3\) is given by \(\alpha = dz + y\,dx\) (see \cref{fig:ReebAndContactStructure}).

\begin{figure}[h]
    \centering
    \includegraphics[width = .45\textwidth]{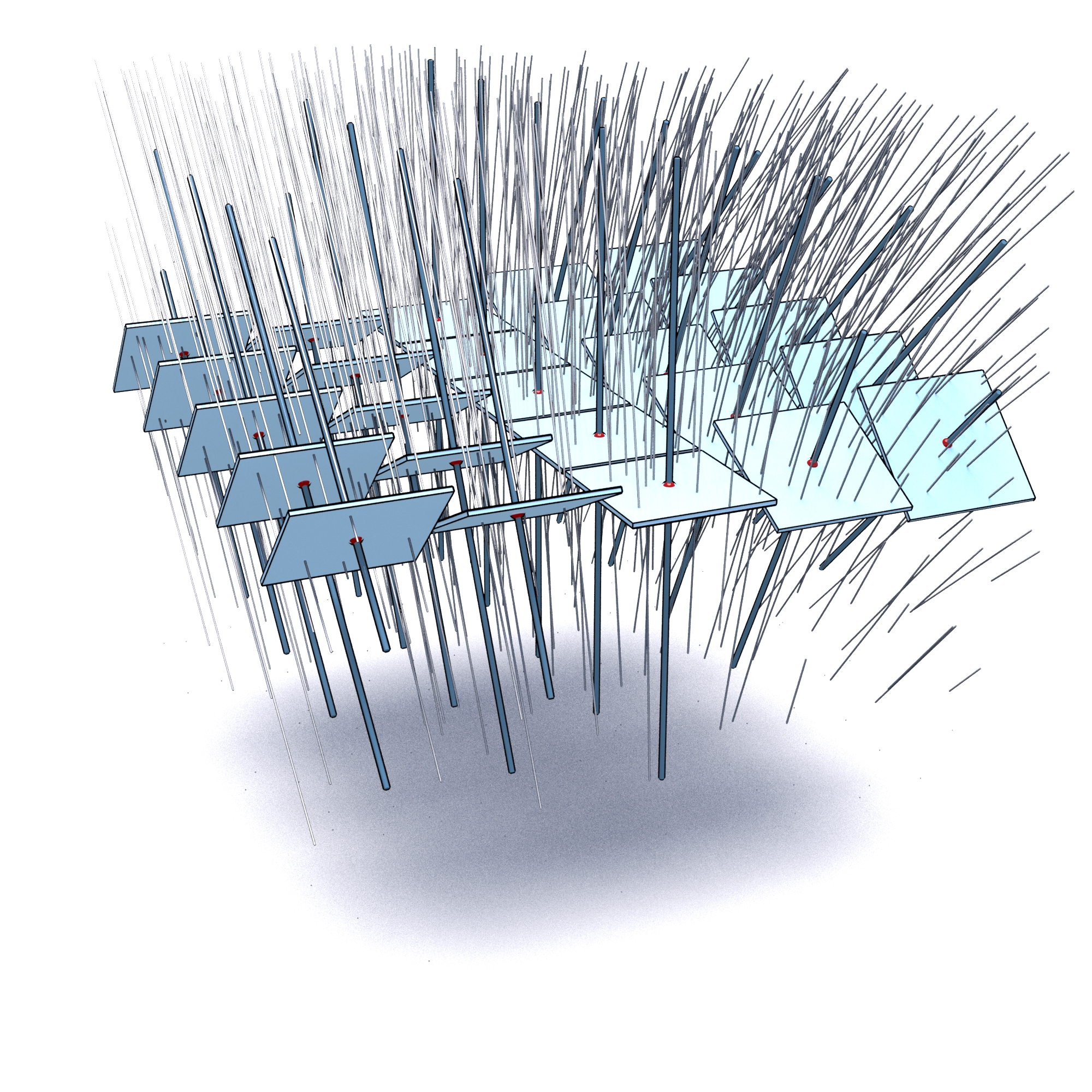}
    \includegraphics[width = .45\textwidth]{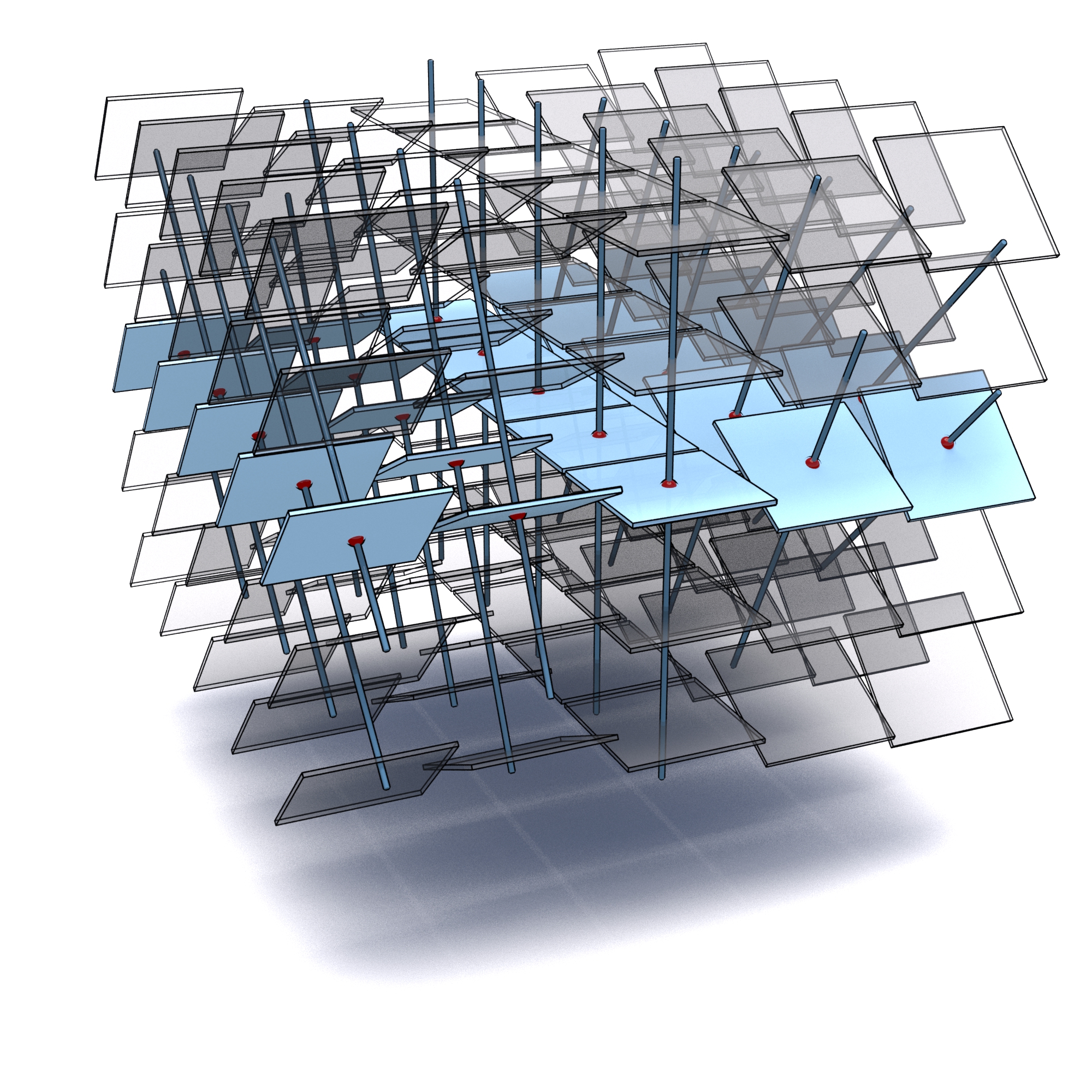}
    \caption{Field lines of the Reeb vector field corresponding to the contact 1-form \(\alpha = dz + y\,dx\) (left) and the corresponding contact structure of the contact manifold \((\RR^3,\ker\alpha)\).}
    \label{fig:ReebAndContactStructure}
\end{figure}

\begin{definition}
	 On an orientable manifold \(M\) of odd dimension \(2n+1\) with contact \(1\)-form \(\alpha\in\Omega^1(M)\), the vector field \(X\in\Gamma TM\) uniquely defined by 
	 \begin{align}\label{eq:ReebVectorFieldDefiningEq}
	 	\alpha(X)=1, \qquad X\in\ker(d\alpha).
	 \end{align}
  is called the \emph{Reeb vector field} of the contact \(1\)-form \(\alpha\).
\end{definition}
The Hopf fibration (\cref{exp:HopfFibration}) gives an example of a Reeb vector field of a contact manifold.

The following \cref{thm:ReebIsForceFreeWithNonVanishingF} states that there always exists a metric on a contact manifold \(M\) with respect to which the Reeb vector field of a corresponding contact \(1\)-form is geodesic.
\begin{theorem}[{\cite{Etnyre:2000:CTH}}]\label{thm:ReebIsForceFreeWithNonVanishingF}
	Let \(M\) be an orientable manifold of odd-dimension \(2n+1\) and \(X\in\Gamma TM\). Then \(X\) is the Reeb vector field of a contact structure \(\alpha\in\Omega^1(M)\) if and only if there is a Riemannian metric \(g\) on \(M\) such that \(X\) force-free with non-vanishing proportionality factor.
\end{theorem}
In fact, \(X\) is of unit length with respect to the relevant metric. Hence, as a consequence of \corref{thm:UnitBetaBeltramiGeodesic} we conclude
\begin{corollary}
    In the setting of \cref{thm:ReebIsForceFreeWithNonVanishingF} the vector field \(X\) it is moreover geodesic with respect to said metric \(g\).
\end{corollary}

\subsection{Killing Vector Fields}
Let \(M\) be an \(n\)-dimensional Riemannian manifold with Riemannian metric \(g\). Then a vector field which generates an isometric flow, \ie, an infinitesimal isometry of \(M\) is called a \emph{Killing vector field} (\cref{fig:KillingVector Fields}). Examples include, \eg, vector fields associated to rigid body transformations in \(\RR^n\).
\begin{definition}
    \label{def:Killing}
    A vector field \(B\in\Gamma TM\) on a Riemannian manifold \(M\) is called a \emph{Killing vector field} if \[\LD_Bg=0.\]
    A flux form \(\beta\in\Omega^{n-1}(M)\) is called \emph{Killing} if the associated vector field is a Killing vector field.
\end{definition}

\begin{figure}[h]
   \centering
   \includegraphics[width=.45\columnwidth]{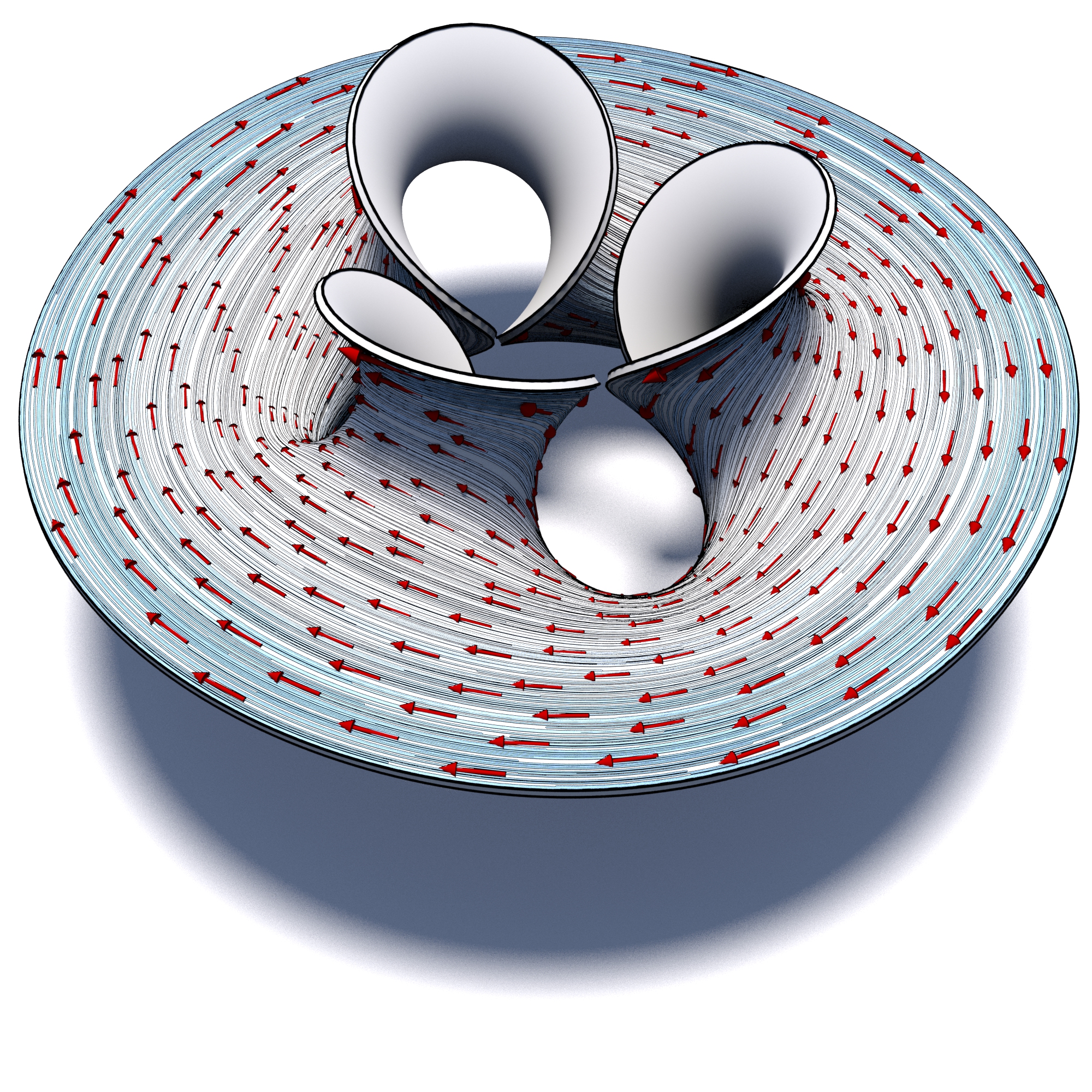}
   \includegraphics[width=.45\columnwidth]{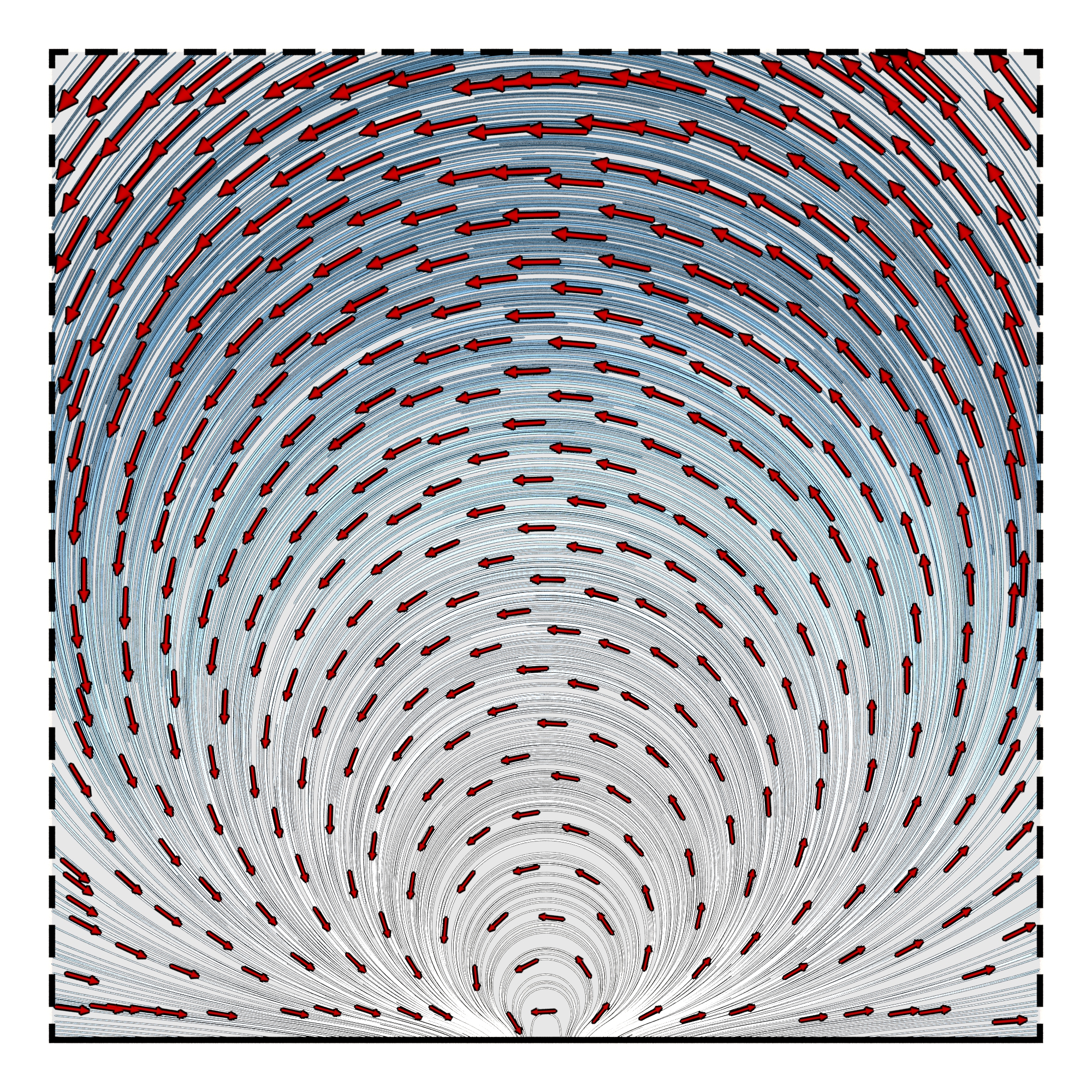}
   \caption{\label{fig:KillingVector Fields}A Killing field and the associated flow lines on an \emph{Enneper surface} (left) and on a piece of the hyperbolic plane in the upper half-plane model (right).}
\end{figure}
\begin{proposition}
    \label{thm:prop}
    A vector field \(B\in\Gamma TM\) on a Riemannian manifold \(M\) is a Killing vector field if and only if for \(Y,Z\in\Gamma TM\), \[g(\nabla_YB,Z)=-g(Y,\nabla_ZB).\]
\end{proposition}
Note that \cref{thm:UnitBetaBeltramiGeodesic} is true as long as the flux form \(\beta\)  has constant (not necessarily unit) norm, which shows that flux forms of constant norm are special. In this section we derive even more interesting consequences of the constancy of the norm. On Riemannian manifolds, Killing vector fields of constant length are known to be related to geodesic foliations~\cite{berestovskii2008killing}. 
\begin{lemma}%[{\cite[Prop. 1]{berestovskii2008killing}}]
    \label{thm:ConstKillingIffGeodesic}
    A Killing vector field \(B\in\Gamma TM\) on a Riemannian manifold \(M\) has constant length with respect to the metric \(g\) if and only if it is geodesic with respect to \(g\).
\end{lemma}
\begin{proof}
    By \cref{thm:prop},
    \begin{align}
        \label{eq:ConstKillingIffGeodesic}
        dg(B,B) = 2g(\nabla B,B) = -2(\nabla_BB)^\flat,
    \end{align}
    from which the claim immediately follows.
\end{proof}

\begin{remark}
Conditions on the curvature of the manifold \(M\) need to be satisfied for the converse statement of \cref{thm:ConstKillingIffGeodesic}, that is for when a geodesic vector field of constant length is Killing, are given in \cite{Deshmukh:2019:GVF}.
\end{remark}

The Hopf field (\cref{fig:HopfFibration}) also serves as an example for \cref{thm:ConstKillingIffGeodesic}. With its unit norm geodesic field lines, it is not only force-free (\cref{thm:UnitBetaBeltramiGeodesic}), but also a Killing vector field on \(S^3\). We can use \cref{thm:ConstKillingIffGeodesic} to show that---even without a constant norm---Killing vector fields are in fact conformally geodesic vector fields.
\begin{theorem}\label{thm:KillingImpliesUnitGeodesic}
	Let \(M\) be an \(n\)-dimensional manifold with Riemannian metric \(g\) and \(B\in\Gamma TM\) a Killing vector field. Then there is a Riemannian metric \(h\in[g]\) such that \(B\) is geodesic and of unit length.
\end{theorem}
\begin{proof}
	From \cref{thm:prop} we conclude that \(g(\nabla_BB,B)=0\). Define \(h\coloneqq e^{-2u}\,g\) for \(e^{2u}\coloneqq g(B,B)\), then \[d_Be^{-2u} = -2g(B,B)^{-2}g(\nabla_BB,B) = 0.\]
	In particular, 
	\[\LD_Bh = \LD_B(e^{-2u}g) = d_Be^{-2u}g + e^{-2u}\LD_Bg = 0.\]  
	Therefore, \(B\) is also a Killing vector field with respect to \(h\) and in particular, \(h(B,B)= \frac{1}{g(B,B)}g(B,B)=1\). The claim now follows from \cref{thm:ConstKillingIffGeodesic}.
\end{proof}
\begin{corollary}
    Let \(M\) be an \(n\)-dimensional manifold with Riemannian metric \(g\) and \(B\in\Gamma TM\) a Killing vector field. Then there is a Riemannian metric \(h\in[g]\) such that \(B\) is force-free.
\end{corollary}
\begin{proof}
    By \cref{thm:KillingImpliesUnitGeodesic} there is a conformally equivalent metric \(h\) on \(M\) with respect to which \(B\) is geodesic and of unit length. We note that, since the flow induced by a Killing field \(B\) preserves \(h\) the same holds true for the induced volume form, \ie, \(B\) is volume preserving with respect to the volume form induced by \(h\). Then, by \cref{thm:UnitBetaBeltramiGeodesic} \(B\) is the force-free.
\end{proof}

\section*{Acknowledgments}
%We would like to acknowledge the assistance of volunteers in putting
%together this example manuscript and supplement.
This work was funded by the Deutsche Forschungsgemeinschaft (DFG - German Research Foundation) - Project-ID 195170736 - TRR109 ``Discretization in Geometry and Dynamics'' and the National Science Foundation - CAREER Award 223906.
Additional support was provided by SideFX software.
The research was conducted during a visiting stay of the second author at California Institute of Technology hosted by Prof.\@ Peter Schr\"oder.
The authors would also like to thank Prof.\@ Ulrich Pinkall, Dr.\@ Felix Kn\"oppel, Mark Gillespie and Sadashige Ishida for initial discussions.

\bibliographystyle{siamplain}
\bibliography{references}

\end{document}